\documentclass[11pt]{amsart}

\usepackage{epigamath}

\usepackage[english]{babel}

\usepackage[all]{xy}
\usepackage{amsmath,amsthm, amssymb, amsgen,amsxtra,amsfonts, amsbsy} 
\usepackage[all]{xy}
\usepackage{enumitem,array}

\usepackage{graphicx}
\usepackage{caption,rotating}
\usepackage{color}
\usepackage{subcaption}
\newcommand{\mathds}[1]{{\mathbb #1}}

\theoremstyle{definition}
\newtheorem{Definition}{Definition}[section]

\newtheorem{Example}[Definition]{Example}

\newtheorem{Remark}[Definition]{Remark}

\newtheorem{Caution}[Definition]{Caution}

\theoremstyle{plain}
\newtheorem{Theorem}[Definition]{Theorem}

\newtheorem{Proposition}[Definition]{Proposition}

\newtheorem{Lemma}[Definition]{Lemma}
\newtheorem{Corollary}[Definition]{Corollary}

\newtheoremstyle{voiditstyle}{3pt}{3pt}{\itshape}{\parindent}%
{\bfseries}{.}{ }{\thmnote{#3}}%
\theoremstyle{voiditstyle}

\newtheoremstyle{voidromstyle}{3pt}{3pt}{\rm}{\parindent}%
{\bfseries}{.}{ }{\thmnote{#3}}%
\theoremstyle{voidromstyle}

\newcommand{\cal}{\mathcal}

\newcommand{\GG}{{\mathds{G}}}

\newcommand{\ZZ}{{\mathds{Z}}}

\newcommand{\OO}{{\cal O}}

\newcommand{\Spec}{{\rm Spec}\:}

\newcommand{\Aut}{{\rm Aut}}

  \newcommand{\red}{{\rm red}}
\EpigaVolumeYear{6}{2022} \EpigaArticleNr{9} 
\ReceivedOn{March 31, 2021}
\InFinalFormOn{October 23, 2021}
\AcceptedOn{February 15, 2022}

\title{Automorphism group schemes \\of bielliptic and quasi-bielliptic surfaces}
\titlemark{Automorphism group schemes of bielliptic and quasi-bielliptic surfaces}

\author{Gebhard Martin}
\address{Mathematisches Institut der Universit\"at Bonn, Endenicher Allee 60, 53115 Bonn, Germany}
\email{gmartin@math.uni-bonn.de}

\authormark{G. Martin}

\AbstractInEnglish{Bielliptic and quasi-bielliptic surfaces form one of the four classes of minimal smooth projective surfaces of Kodaira dimension $0$. In this article, we determine the automorphism group schemes of these surfaces over algebraically closed fields of arbitrary characteristic, generalizing work of Bennett and Miranda over the complex numbers; we also find some cases that are missing from the classification of automorphism groups of bielliptic surfaces in characteristic $0$.}



\acknowledgement{Research of the author is funded by the DFG Research Grant MA 8510/1-1 ``Infinitesimal Automorphisms of Algebraic Varieties''.}

\begin{document}



\maketitle

\begin{prelims}

\DisplayAbstractInEnglish

\bigskip

\DisplayKeyWords

\medskip

\DisplayMSCclass







\end{prelims}


\newpage

\setcounter{tocdepth}{1}

\tableofcontents


\section{Introduction}

We are working over an algebraically closed field $k$ of characteristic $p \geq 0$.
Bielliptic and quasi-bielliptic surfaces form one of the four types of minimal smooth projective surfaces of Kodaira dimension $0$. Each bielliptic surface $X$ is a quotient $\pi: E \times C \to (E \times C)/G = X$, where $E$ and $C$ are elliptic curves and $G \subseteq E$ is a finite subgroup scheme of $E$ that acts faithfully on $C$ via
$
\alpha: G \to \Aut_{C}.
$
Moreover, the image of $\alpha$ is not entirely contained in the group of translations $C$. This latter condition guarantees that $X$ is not an Abelian surface. All possible combinations of $E,C,G$ and $\alpha$ have been determined: if $p = 0$ by Bagnera and de Franchis in \cite{Bagnera}, and if $p \neq 0$ by Bombieri and Mumford in \cite{BombieriMumford2}.

Similarly, quasi-bielliptic surfaces, which exist if and only if $p \in \{2,3\}$, are obtained by replacing $C$ by a cuspidal plane cubic curve and by imposing on $\alpha$ the condition that the cusp of $C$ is not a fixed point of the group scheme $\alpha(G)$. As in the bielliptic case, it is possible to determine all combinations of $E,C,G$ and $\alpha$. We refer the reader to \cite{BombieriMumford3}, but note that not all cases listed there actually occur (see Remark \ref{rem: char3fall} and Remark \ref{rem: char2lambda}).

Bielliptic and quasi-bielliptic surfaces come with two natural fibrations: one of them is the Albanese map $f_E: X \to E/G =: E'$, which is quasi-elliptic if $X$ is quasi-bielliptic, and elliptic if $X$ is bielliptic. All closed fibers of $f_E$ are isomorphic to $C$, since this holds after pulling back along the faithfully flat morphism $E \to E/G$. The second fibration $f_C: X \to C/\alpha(G) =: C' \cong \mathbb{P}^1$ is always elliptic, but has multiple fibers. 

The purpose of this article is to determine the automorphism group scheme $\Aut_X$ of $X$. If $p = 0$, this has been carried out by Bennett and Miranda in \cite{BennettMiranda}. By Proposition \ref{prop: normalizer}, the actions of the centralizers $C_{\Aut_E}(G)$ and $C_{\Aut_C}(\alpha(G))$ on the first and second factor of $E \times C$, respectively, descend to $X$ and we consider them as subgroup schemes of $\Aut_X$ via these actions.
 Then, the following theorem is the key result of this article.

\begin{Theorem} \label{Main}
Let $X = (E \times C)/G$ be a bielliptic or quasi-bielliptic surface. Then, there is a short exact sequence of group schemes
$$
1 \to (C_{\Aut_E}(G) \times C_{\Aut_C}(\alpha(G)))/G \overset{\pi_*}{\to} \Aut_X \to M \to 1,
$$
where $G$ is embedded via ${\rm id} \times \alpha$ and $M$ is a finite and \'etale group scheme. In particular, $\Aut_X$ is of finite type.
\end{Theorem}

We refer the reader to Theorem \ref{Mainwithdetails} for a refined statement including a description of the group schemes $\Aut_{X/E'}$ and $\Aut_{X/C'}$ of automorphisms of $X$ over $E'$ and over $C'$, respectively. While the part of $\Aut_X$ coming from the centralizers is straightforward to calculate and understand, the part $M$ is more elusive. In particular, we note that $M$ can be non-trivial even in characteristic $0$, contrary to what is claimed in \cite[Section 2]{BennettMiranda}. Even though we do not see an a priori reason for this, $M$ always comes from automorphisms of $E \times C$. Then, a posteriori, Theorem \ref{Mainwithdetails} \eqref{item: Main6} gives a complete description of $M$ (see Remark \ref{rem: M}). By Proposition \ref{prop: normalizer} and Lemma \ref{lem: nonconnecteddiagonal} \eqref{item: nonconnecteddiagonal3}, we have the following corollary of our analysis.

\begin{Corollary}
Let $X = (E \times C)/G$ be a bielliptic or quasi-bielliptic surface. Then, 
$$
\Aut_X \cong N_{\Aut_{E} \times \Aut_{C}}(G)/G,
$$
where $N_{\Aut_{E} \times \Aut_{C}}(G)$ is the normalizer of $G$ in $\Aut_E \times \Aut_C$.
\end{Corollary}

By Corollary \ref{cor: autx0}, we have $E \cong (\Aut_X^\circ)_{\red}$ and $(\Aut_X^\circ)_{\red}$ is normal in $\Aut_X$, so we can write the quotient $\Aut_X/E$ as an extension of $M$ by $(C_{\Aut_E}(G)/E) \times (C_{\Aut_C}(\alpha(G))/\alpha(G))$. These group schemes can be calculated explicitly and this will be carried out in Section \ref{sec: normalizercentralizer}.
 In the following Tables \ref{bielliptic}, \ref{quasibielliptic3}, and \ref{quasibielliptic2}, the groups $S_n,A_n$, and $D_{2n}$ are the symmetric, alternating, and dihedral group (of order $2n$), respectively, and $M_2$ is the $p$-torsion subscheme of a supersingular elliptic curve. The stars and daggers in Table \ref{bielliptic} will be explained in Remark \ref{BennettMiranda}.

\begin{Corollary}\label{cormain}
Let $X = (E \times C)/G$ be a bielliptic or quasi-bielliptic surface. Then, depending on the group scheme $G$ and the $j$-invariants $j(E)$ and $j(C)$, the group schemes $C_{\Aut_E}(G)/E$, $C_{\Aut_C}(\alpha(G))/\alpha(G)$ and $M$ are as in Table \ref{bielliptic}, \ref{quasibielliptic3}, and \ref{quasibielliptic2}.

\begin{table}[!htb]
\centering
\resizebox{\textwidth}{!}{$\displaystyle
\begin{array}{|c|c|c|c|c|c|c|}
\hline
G & 
j(E) & C_{\Aut_E}(G)/E & j(C) & C_{\Aut_C}(\alpha(G))/\alpha(G) & M & p \\ \hline \hline

\ZZ/2\ZZ & 
\begin{array}{cc}
a) & \text{any} \\
b) & 1728^*
\end{array}
&
 \begin{array}{cc}
a) & \ZZ/2\ZZ \\
b) & \ZZ/4\ZZ
\end{array}
 &
  \begin{array}{cc}
i) & \neq 0,1728 \\
ii) & 1728 \\
iii) & 0
\end{array}
& 
 \begin{array}{cc}
i) & (\ZZ/2\ZZ)^2\\
ii) & D_8 \\
iii) & A_4
\end{array}
&
\{1\}&
\neq 2,3 \\ \hline 

(\ZZ/2\ZZ)^2 & \begin{array}{cc}
a) & \text{any} \\
b) & 1728^\dagger
\end{array}
&  \ZZ/2\ZZ&
\begin{array}{cc}
i) & \text{any}\\
ii) & 1728^*\\
\end{array}
& 
\begin{array}{cc}
i) & \ZZ/2\ZZ \\
ii) & \phantom{i}(\ZZ/2\ZZ)^2 \\
\end{array}
& \begin{array}{cc}
a) & \{1\} \\
b) & \ZZ/2\ZZ
\end{array}
& \neq 2,3 \\ \hline 

\ZZ/3\ZZ & 

\begin{array}{cc}
a) & \text{any}\\
b) & 0^*\\
\end{array}
 & 
 \begin{array}{cc}
a) & \{1\} \\
b) & \ZZ/3\ZZ\\
\end{array}
& 0 & S_3 & \{1\}& \neq 2,3 \\ \hline 

(\ZZ/3\ZZ)^2 &
\begin{array}{cc}
a) & \text{any} \\
b) & 0^\dagger
\end{array}
 &  \{1\} & 0 & \{1\} & \begin{array}{cc}
a) & \{1\} \\
b) &  \ZZ/3\ZZ
\end{array}
 & \neq 2,3 \\ \hline  

\ZZ/4\ZZ & any  &  \{1\} & 1728 & \ZZ/2\ZZ & \{1\} & \neq 2 \\ \hline 

\ZZ/4\ZZ \times \ZZ/2\ZZ & any  &  \{1\} & 1728 & \{1\} &  \{1\} & \neq 2 \\ \hline  

\ZZ/6\ZZ &  any &  \{1\} & 0 & \{1\}  & \{1\} & \neq 2,3 \\ \hline \hline

\ZZ/2\ZZ & 
 \begin{array}{cc}
a) & \neq 0\\
b) & 0 \\
\end{array}
 &
  \begin{array}{cc}
a) & \ZZ/2\ZZ\\
b) & \ZZ/4\ZZ \\
\end{array}
&
  \begin{array}{cc}
i) & \neq 0\\
ii) & 0\\
\end{array}
&
  \begin{array}{cc}
i) & (\ZZ/2\ZZ)^2 \\
ii) & (\ZZ/2\ZZ)^2 \rtimes S_3\\
\end{array}
& \{1\}
& 3 \\ \hline 

(\ZZ/2\ZZ)^2 & \begin{array}{cc}
a) & \neq 0 \\
b) & 0
\end{array} &  \ZZ/2\ZZ& 

\begin{array}{cc}
i) & \neq 0\\
ii) & 0\\
\end{array}
& 
\begin{array}{cc}
i) & \ZZ/2\ZZ\\
ii) & \phantom{ \rtimes d} (\ZZ/2\ZZ)^2 \phantom{S_3}\\
\end{array}
& \begin{array}{cc}
a) & \{1\} \\
b) & \ZZ/2\ZZ
\end{array}& 3 \\ \hline 

\ZZ/3\ZZ & \neq 0 &  \{1\} & 0 & \alpha_3 \rtimes \ZZ/2\ZZ & \ZZ/2\ZZ& 3 \\ \hline 

\ZZ/6\ZZ & \neq 0  &  \{1\} & 0 & \{1\}  & \ZZ/2\ZZ & 3 \\ \hline \hline

\ZZ/2\ZZ &  \neq 0 &  \ZZ/2\ZZ & 
\begin{array}{cc}
i) & \neq 0\\
ii) & 0 \\
\end{array}
&
 \begin{array}{cc}
i) & \mu_2  \times \ZZ/2\ZZ\\
ii) & M_2 \rtimes A_4 \\
\end{array}
 & \{1\} & 2 \\ \hline

\mu_2 \times \ZZ/2\ZZ  &  \neq 0 &  \ZZ/2\ZZ & \neq 0& \ZZ/2\ZZ & \{1\}& 2 \\ \hline

\ZZ/3\ZZ & 
\begin{array}{cc}
a) & \neq 0\\
b) & 0\\
\end{array}
 &
 \begin{array}{cc}
a) & \{1\}\\
b) & \ZZ/3\ZZ \\
\end{array}
 & 0 & S_3 & \{1\}& 2 \\ \hline 
 
(\ZZ/3\ZZ)^2 &
\begin{array}{cc}
a) & \neq 0 \\
b) & 0
\end{array}
 &  \{1\} & 0 & \{1\} & \begin{array}{cc}
a) & \{1\} \\
b) &  \ZZ/3\ZZ
\end{array}
 & 2 \\ \hline  

\ZZ/4\ZZ & \neq 0  &  \{1\} & 0 &\alpha_2 & \ZZ/2\ZZ & 2 \\ \hline

\ZZ/6\ZZ & \neq 0  &  \{1\} & 0 & \{1\}  & \{1\} & 2 \\
\hline 
\end{array}
$}
\caption{Automorphism group schemes of bielliptic surfaces}
\label{bielliptic}
\end{table}

\begin{table}[!htb]
\centering
$$
\begin{array}{|c|c|c|c|c|c|}
\hline
G & j(E) & C_{\Aut_E}(G)/E & C_{\Aut_C}(\alpha(G))/\alpha(G) & M & p  \\ \hline \hline
\mu_3 &  \neq 0 &  \{1\}& S_3 & \{1\}  &3 \\ \hline
\mu_3 \times \ZZ/2\ZZ &  \neq 0 &  \{1\} & \{1\} & \{1\}  &3 \\ \hline
\mu_3 \times \ZZ/3\ZZ &  \neq 0 &  \{1\}& \{1\} & \{1\}  &3  \\ \hline 
\alpha_3 &  0 &  \ZZ/3\ZZ& \alpha_3 \rtimes \ZZ/2\ZZ & \ZZ/4\ZZ  & 3 \\ \hline 
\alpha_3 \times \ZZ/2\ZZ &  0 &  \{1\}& \{1\} & \ZZ/4\ZZ  & 3 \\ \hline
\end{array}
$$
\caption{Automorphism group schemes of \mbox{quasi-bielliptic surfaces in characteristic $3$}}
\label{quasibielliptic3}
\end{table}

\begin{table}[!htb]
\centering
$$
\begin{array}{|c|c|c|c|c|c|c|}
\hline
G & j(E) & C_{\Aut_E}(G)/E & \lambda & C_{\Aut_C}(\alpha(G))/\alpha(G) & M & p \\\hline \hline

\mu_2 &  \neq 0 &  \ZZ/2\ZZ & \begin{array}{cc} i) & \neq 0 \\  ii) & 0 \end{array} &\begin{array}{cc} i) & (\ZZ/2\ZZ)^2\\  ii) & A_4 \end{array} & \begin{array}{cc} i) & \{1\} \\  ii) &  \{1\} \end{array} & 2 \\ \hline

\mu_2 \times \ZZ/3\ZZ &  \neq 0 &  \{1\} & - &\{1\} & \{1\} & 2 \\ \hline 
\mu_2 \times \ZZ/2\ZZ &  \neq 0 &  \ZZ/2\ZZ & any & \ZZ/2\ZZ & \{1\} & 2 \\ \hline 
\mu_4 &  \neq 0 &  \{1\} & - & \ZZ/2\ZZ & \{1\} & 2 \\ \hline 
\mu_4 \times \ZZ/2\ZZ &  \neq 0 &  \{1\} & -& \{1\} & \{1\} & 2 \\ \hline 
\alpha_2 &  0 &  Q_8 & \begin{array}{cc} i) & 1 \\  ii) & 0 \end{array}  &\begin{array}{cc} i) & \alpha_2^2 \rtimes \ZZ/2\ZZ\\  ii)  & (\alpha_4 \rtimes \alpha_4) \rtimes \ZZ/3\ZZ \end{array} &  \begin{array}{cc} i) & \{1\}  \\ ii) & \ZZ/3\ZZ \end{array}& 2 \\ \hline 
\alpha_2 \times \ZZ/3\ZZ &  0 &  \{1\} & -  &\{1\} & \ZZ/3\ZZ & 2 \\ \hline
M_2 &  0 &  \ZZ/2\ZZ &  \neq 0 & \alpha_2 \times \ZZ/2\ZZ & (\ZZ/2\ZZ)^2 &  2 \\ \hline 
\end{array}
$$
\caption{Automorphism group schemes of \mbox{quasi-bielliptic surfaces in characteristic $2$}}
\label{quasibielliptic2}
\end{table}
\end{Corollary}


\begin{Remark} \label{BennettMiranda}
Let us explain the meaning of the stars and daggers in Table \ref{bielliptic}. We denote by $O \in E$ the neutral element with respect to the group law on $E$:
\begin{itemize}
\item Stars: If $p \neq 2,3$ and $j(E) = 1728$, then every automorphism $h_E$ of $(E,O)$ of order $4$ fixes a unique cyclic subgroup of $E$ of order $2$. Similarly, if $p \neq 2,3$ and $j(E) = 0$, then every automorphism $h_E$ of $(E,O)$ of order $3$ fixes a unique cyclic subgroup of $E$ of order $3$. A star after a $j$-invariant in Table  \ref{bielliptic} denotes that the translation subgroup of $G$ or $\alpha(G)$ coincides with this cyclic subgroup. By Lemma \ref{translation}, this implies that $h_E$ is in the corresponding centralizer. We note that such special $2$ and $3$-torsion points do not exist if $p = 2,3$, because $(E,O)$ has more automorphisms in these characteristics.
\item Daggers: A dagger after $j(E)$ denotes that the special $2$ or $3$-torsion points described above maps to a translation in $\alpha(G)$. In these cases, the automorphism $(h_E,h_C)$, where $h_E$ is an automorphism of order $4$ or $3$ of $(E,O)$ and $h_C$ is translation by a suitable $4$ or $3$-torsion point, respectively, normalizes the $G$-action on $E \times C$ and hence descends to $X$. Since $(h_E,h_C)$ does not centralize the $G$-action, it induces a non-trivial element of $M$. See the proof of Proposition \ref{prop: trivialM} for a precise description of the automorphism $(h_E,h_C)$ in these cases.
\end{itemize}
These cases seem to be missing from \cite{BennettMiranda}, since they were not listed in \cite[Table 1.1]{BennettMiranda}, which is why \cite[Table 3.2]{BennettMiranda} differs from our Table \ref{bielliptic}. 
\end{Remark}

\begin{Remark}
In the quasi-bielliptic case in characteristic $2$, the action of $G$ on $E \times C$ sometimes depends on a parameter $\lambda \in k$ and so does $\Aut_X$. For an explicit description of $\lambda$, see Section \ref{quellchar2}. The parameter $\lambda$ should be thought of as a replacement for the $j$-invariant of the curve $C$.
\end{Remark}

Recall that the space $H^0(X,T_X)$ is the tangent space of $\Aut_X$ at the identity. Since $E \cong (\Aut_X^\circ)_{\red}$, $\Aut_X$ is smooth if and only if $h^0(X,T_X) = 1$. A careful inspection of Tables \ref{bielliptic}, \ref{quasibielliptic3}, and \ref{quasibielliptic2}, and of the orders of the canonical bundle $\omega_X$ determined in \cite{BombieriMumford2} and \cite{BombieriMumford3} shows the following.

\begin{Corollary}
Let $X$ be a bielliptic or quasi-bielliptic surface. Then, the following hold:
\begin{enumerate}
\item $h^0(X,T_X) \leq 3$.
\item If $X$ is bielliptic or $p \neq 2$, then $h^0(X,T_X) \leq 2$.
\item $h^0(X,T_X) = 1$ if and only if $\omega_X \not \cong \OO_X$ if and only if $\Aut_X$ is smooth.
\end{enumerate}
\end{Corollary}

\subsection*{Acknowledgements}
I would like to thank Daniel Boada de Narv\'aez, Christian Liedtke, and Claudia Stadlmayr for helpful comments on a first version of this article and Curtis Bennett and Rick Miranda for interesting discussions. I am indebted to Azniv Kasparian and Gregory Sankaran for pointing out a mistake in an earlier version of this article. I am grateful to the anonymous referee for thorough comments and very helpful suggestions that helped me to improve the exposition and to fix inaccuracies.
Finally, I would like to thank the Department of Mathematics at the
University of Utah for its hospitality while this article was written.

\section{Notation and generalities on automorphism group schemes}
\label{sec: notation}
Let $\pi: Y \to X$ be a morphism of proper varieties over an algebraically closed field $k$. There are several $k$-group schemes of automorphisms associated to $\pi$. We follow the notation of \cite[Section 2.4]{BrionNotes}, which we recall for the convenience of the reader. Throughout, $T$ is an arbitrary $k$-scheme.

\begin{itemize}
\item The \emph{automorphism group scheme} $\Aut_X$ \emph{of} $X$ is the $k$-group scheme whose group of $T$-valued points $\Aut_X(T) := \Aut(X \times_k T)$ is the group of automorphisms of $X_T := X \times_k T$ over $T$. By \cite[Theorem (3.7)]{MatsumuraOort}, $\Aut_X$ is a group scheme locally of finite type over $k$. The identity component of $\Aut_X$ is denoted by $\Aut_X^\circ$.
\item The \emph{automorphism group scheme} $\Aut_{\pi}$ \emph{of the morphism} $\pi$ is the $k$-group scheme such that $\Aut_{\pi}(T)$ consists of pairs $(g,h) \in \Aut_Y(T) \times \Aut_X(T)$ making the diagram
$$
\xymatrix{
Y_T \ar[r]^g \ar[d]^{\pi_T} & Y_T \ar[d]^{\pi_T} \\
X_T \ar[r]^{h} & X_T
}
$$
commutative. In particular, $\Aut_\pi(-)$ is a closed subfunctor of $\Aut_Y(-) \times \Aut_X(-)$, hence $\Aut_\pi$ is representable by a group scheme locally of finite type over $k$.
\item The group scheme $\Aut_{\pi}$ comes with projections to $\Aut_Y$ and $\Aut_X$.
If $\pi$ is faithfully flat, then the first projection $\Aut_{\pi} \to \Aut_Y$ is a closed immersion and we will use this to consider $\Aut_{\pi}$ as a subgroup scheme of $\Aut_Y$. We denote the second projection by $\pi_*: \Aut_{\pi} \to \Aut_X$.
\item The \emph{automorphism group scheme} $\Aut_{Y/X}$ \emph{of} $Y$ \emph{over} $X$ is the $k$-group scheme whose group of $T$-valued points $\Aut_{Y/X}(T)$ consists of automorphisms $g \in \Aut_{Y}(T)$ such that $\pi_T \circ g = \pi_T$. 
By definition, there is an exact sequence realizing $\Aut_{Y/X}$ as a subgroup scheme of $\Aut_{\pi}$:
\begin{equation} \label{eq: normalizer}
1 \to \Aut_{Y/X} \to \Aut_{\pi} \overset{\pi_*}{\to} \Aut_X
\end{equation}
\item Given a closed subgroup scheme $G \subseteq \Aut_Y$, the \emph{normalizer} $N_{\Aut_Y}(G)$ \emph{of} $G$ \emph{in} $\Aut_Y$ is the $k$-group scheme whose group of $T$-valued points is
$$
\hspace{12mm} N_{\Aut_Y}(G)(T) = \{ h \in \Aut_Y(T) \mid h_{T'} \circ g \circ (h_{T'})^{-1} \in G(T') \text{ for all } T' \to T \text{ and } g \in G(T') \}.
$$
The \emph{centralizer} $C_{\Aut_Y}(G)$ \emph{of} $G$ \emph{in} $\Aut_Y$ is the group scheme whose $T$-valued points satisfy the stronger condition $h_{T'} \circ g \circ (h_{T'})^{-1} = g$ instead. 
By \cite[Expos\'e VIB, Proposition 6.2 (iv)]{SGA3}, both $N_{\Aut_Y}(G)$ and $C_{\Aut_Y}(G)$ are closed subgroup schemes of $\Aut_Y$.
\end{itemize}

\begin{Caution}
The notation $\Aut_{Y/X}$ is also a standard notation for the group functor on the category of $X$-schemes that associates to an $X$-scheme $Z$ the automorphism group of $Y \times_X Z$ over $Z$. Since these relative automorphism group functors do not occur in this article, we decided to use the notation introduced above instead of more cumbersome, albeit more precise, notation such as $\Aut_{Y/X/k}$.
\end{Caution}

\section{Automorphism group schemes of quotients}
\label{sec: quotient}
In this section, we study Sequence \eqref{eq: normalizer} in the case where $\pi: Y \to X$ is a finite quotient.

\begin{Proposition} \label{prop: normalizer}
If $G$ is a finite group scheme acting freely on a proper variety $Y$ such that the geometric quotient $\pi:Y \to Y/G =: X$ exists as a scheme, then we have $\Aut_{Y/X} = G$ and $\Aut_{\pi} = N_{\Aut_Y}(G)$ as subgroup schemes of $\Aut_Y$. In particular, Sequence \eqref{eq: normalizer} becomes
\begin{equation*}
1 \to G \to N_{\Aut_Y}(G) \overset{\pi_*}{\to} \Aut_X.
\end{equation*}
\end{Proposition}

\proof
First, we show that $\Aut_{Y/X} = G$. By \cite[Lemma 4.1]{BrionTorsor}, there is a $G$-equivariant isomorphism $\Aut_{Y/X} \cong Hom(Y,G)$, where the $T$-valued points of the latter are ${\rm Hom}(Y \times T,G)$ and $G$ is embedded as $G = Hom(\Spec k, G)$.
Since $Y$ is a proper variety and taking global sections commutes with flat base change, we have $H^0(Y \times T,\OO_{Y \times T}) = k \otimes_k H^0(T,\OO_T) = H^0(T,\OO_T)$ for every affine $k$-scheme $T$. As $G$ is affine, this implies ${\rm Hom}(Y \times T,G) = {\rm Hom}(T,G) = G(T)$, which is what we had to show. 

Next, we show $\Aut_{\pi} = N_{\Aut_Y}(G)$. For this, let $h \in \Aut_Y(T)$ be an automorphism of $Y_T$. Then, $h \in \Aut_{\pi}(T)$ if and only if there is $h' \in \Aut_{X}(T)$ such that the following diagram commutes
$$
\xymatrix{
Y_T \ar[r]^h \ar[d]^{\pi_T} & Y_T \ar[d]^{\pi_T} \\
X_T \ar[r]^{h'} & X_T.
}
$$
Comparing degrees, it is easy to check that the geometric quotient of $Y_T$ by the induced free action of $G$ coincides with $\pi_T$, so the morphism $\pi: Y \to X$ is a universal geometric quotient of $Y$, hence also a universal categorical quotient by \cite[Proposition 0.1]{MumfordGIT}. Therefore, the automorphism $h'$ exists if and only if $\pi_T \circ h$ is $G$-invariant, that is, if and only if for every $T$-scheme $T'$ we have $\pi_{T'} \circ h_{T'} \circ g = \pi_{T'} \circ h_{T'}$ for all $g \in G(T')$. This is equivalent to $h_{T'} \circ g \circ h_{T'}^{-1} \in \Aut_{Y/X}(T') = G(T')$ for all $g \in G(T')$, which is precisely the condition that $h \in N_{\Aut_Y}(G)$.
\qed

\begin{Example} \label{ex: normalizer}
Contrary to the situation for abstract groups, Proposition \ref{prop: normalizer} typically fails if $Y$ is a non-proper variety or the action of $G$ is not free. Indeed, consider any infinitesimal subgroup scheme $G \subseteq {\rm PGL}_2$ of length $p$. The $k$-linear Frobenius $F: Y := \mathbb{P}^1 \to \mathbb{P}^1 =: X$ is the geometric quotient for the action of $G$ on $\mathbb{P}^1$ and $\Aut_{Y/X} =  {\rm PGL}_2[F]$ is the kernel of Frobenius on ${\rm PGL}_2$. Moreover, we have $\Aut_F = {\rm PGL}_2$. Thus, $\Aut_{Y/X}$ and $\Aut_F$ are strictly bigger than $G$ and $N_{{\rm PGL}_2}(G)$ even though $F$ is a $G$-torsor over an open subscheme of~$X$. 
\end{Example}

Even though Example \ref{ex: normalizer} shows that Proposition \ref{prop: normalizer} fails for non-free actions on curves, we can at least describe the $k$-rational points in Sequence \eqref{eq: normalizer} if the quotient is smooth.

\begin{Proposition} \label{prop: normalizerrationalpoints}
Let $G$ be a finite group scheme acting faithfully on a proper integral curve $D$ with geometric quotient $\varphi: D \to D':=D/G$. Assume that $D'$ is smooth. Then, we have
\[\Aut_{D/D'}(k) = G(k)\quad \text{and} \quad\Aut_{\varphi}(k) = N_{\Aut(D)}(G(k)).\]
\end{Proposition}

\begin{proof}
We can consider the four groups as subgroups of the group $\Aut_k(k(D))$ of $k$-linear field automorphisms of $k(D)$ via the injective restriction map $\Aut(D) \hookrightarrow \Aut_k(k(D))$. We have a tower of field extensions $k(D') \subseteq k(D)^{G(k)} \subseteq k(D)$, where $k(D') \subseteq k(D)^{G(k)}$ is purely inseparable and  $k(D)^{G(k)} \subseteq k(D)$ is a Galois extension with Galois group $G(k)$. An elementary calculation shows that $N_{\Aut_k(k(D))}(G(k))$ is the subgroup of $\Aut_k(k(D))$ of automorphisms preserving $k(D)^{G(k)}$. Since $D'$ is a curve, we have $k(D') = (k(D)^{G(k)})^{p^n}$ for some $n \geq 0$, so an automorphism of $k(D)$ preserves $k(D)^{G(k)}$ if and only if it preserves $k(D')$. Hence, $N_{\Aut_k(k(D))}(G(k))$ is also the group of automorphisms of $k(D)$ preserving $k(D')$. On the other hand, since $D'$ is smooth and proper, $\Aut_{\varphi}(k)$ consists precisely of those automorphisms of $D$ which, when restricted to $k(D)$, preserve $k(D')$. Hence, we have
$
N_{\Aut(D)}(G(k)) = N_{\Aut_k(k(D))}(G(k)) \cap \Aut(D) = \Aut_{\varphi}(k),
$
and $\Aut_{D/D'}(k) = \Aut_{k(D')}(k(D)) \cap \Aut(D) = G(k)$, which is what we had to show.
\end{proof}

\section{Proof of Theorem \ref{Main}}
\label{sec: main}
Throughout this section, $E$ and $C$ are integral curves of arithmetic genus $1$, and we assume that $E$ is smooth and $C$ is either smooth or has a single cusp as singularity. We choose a point $O \in E$ and consider $E$ as an elliptic curve with identity element $O$. We fix a finite subgroup scheme $G \subseteq E$, and a monomorphism $\alpha: G \to \Aut_C$ such that $\alpha(G)$ is not contained in the group of translations of $C$ if $C$ is smooth, and not contained in the stabilizer of the cusp if $C$ is singular.
In particular, the actions of $G$ on $E$ (via translations) and $C$ (via $\alpha$) give rise to a product action of $G$ on $E \times C$ and we set $X := (E \times C)/G$ with quotient map $\pi: E \times C \to X$.
We have the following commutative diagram with two cartesian squares:
\vspace{0.2cm}
\begin{equation*}
\xymatrix{
E \times C \ar[rd]^\pi \ar[d]^{\pi_C} \ar[r]^{\pi_E} & X \times_{E'} E \ar[r] \ar[d] \ar@{}[rd] | {\square} & E \ar[d] \\
X \times_{C'} C \ar[r] \ar[d] \ar@{}[rd] | {\square}& X \ar[r]_{f_E} \ar[d]^{f_C} & E/G =: E' \\
C \ar[r] & C/\alpha(G) =: C'.
}
\end{equation*}
\vspace{0.2cm}Since $G$ acts freely on $E$, the map $\pi_E$ induces isomorphisms on the fibers of $E \times C \to E$ and $X \times_{E'} E \to E$ and thus, as both maps are flat, the morphism $\pi_E$ is an isomorphism. The following lemma shows that the automorphism group scheme of $X$ is controlled by the fibrations $f_E$ and $f_C$.

\begin{Lemma} \label{lem: equivariant}
There is a unique action of $\Aut_X$ on $C'$ and on $E'$ such that both $f_E: X \to E'$ and $f_C: X \to C'$ are $\Aut_X$-equivariant. In particular, there are exact sequences
$$
1 \to \Aut_{X/E'} \to \Aut_X \overset{(f_E)_*}{\to} \Aut_{E'}
$$
and
$$
1 \to \Aut_{X/C'} \to \Aut_X \overset{(f_C)_*}{\to} \Aut_{C'}.
$$
\end{Lemma}
\begin{proof}
The $\Aut_X^\circ$-action on $X$ descends to both $E'$ and $C'$ by Blanchard's Lemma \cite[Proposition 4.2.1]{Brion}. Since $f_E$ and $f_C$ are the only fibrations of $X$ and $E' \not \cong C' \cong \mathbb{P}^1$, it is also clear that the action of the abstract group $\Aut(X)$ descends to $E'$ and $C'$. By \cite[Lemma 2.20 (ii)]{BrionNotes}, this is enough to prove that the whole $\Aut_X$-action descends uniquely to the two curves $E'$ and $C'$.

With respect to the $\Aut_X$-actions of the previous paragraph, we have $\Aut_X = \Aut_{f_C} = \Aut_{f_E}$, hence the short exact sequences in the statement of the lemma are special cases of Sequence \eqref{eq: normalizer}.
\end{proof}

The idea for the proof of Theorem \ref{Main} is to use the isomorphism $\pi_E$ to lift group scheme actions from $X$ to $E \times C$. By Proposition \ref{prop: normalizer}, the automorphisms of $X$ that come from $E \times C$ are induced by the normalizer $N_{\Aut_{E \times C}}(G)$. Therefore, before proving Theorem \ref{Main}, we study $N_{\Aut_{E \times C}}(G)$. 
For the following lemma, note that there is a natural inclusion $\Aut_E \times \Aut_C \hookrightarrow \Aut_{E \times C}$ given by letting $\Aut_E$ and $\Aut_C$ act on the first and second factor, respectively. In particular, we can consider $C_{\Aut_E}(G) \times C_{\Aut_C}(\alpha(G))$ and $N_{\Aut_E}(G) \times N_{\Aut_C}(\alpha(G))$ as subgroup schemes of $\Aut_{E \times C}$.

\begin{Lemma}\label{lem: nonconnecteddiagonal}
The normalizer $N_{\Aut_{E \times C}}(G)$ of $G$ in $\Aut_{E \times C}$ satisfies the following properties:
\begin{enumerate}
\item \label{item: nonconnecteddiagonal1} $N_{\Aut_{E \times C}}(G) \supseteq C_{\Aut_E}(G) \times C_{\Aut_C}(\alpha(G))$.
\item \label{item: nonconnecteddiagonal2}  $N_{\Aut_{E \times C}}^{\circ}(G) = C_{\Aut_E}^\circ(G) \times C_{\Aut_C}^\circ(\alpha(G))$.
\item \label{item: nonconnecteddiagonal3} \mbox{$N_{\Aut_{E \times C}}(G)(T) = \{ (h_E,h_C) \in N_{\Aut_E}(G)(T) \times N_{\Aut_C}(\alpha(G))(T) \mid \alpha_T \circ {\rm ad}_{h_E} = {\rm ad}_{h_C} \circ \alpha_T \},$} where ${\rm ad}_{h_E}$ and ${\rm ad}_{h_C}$ denote conjugation by $h_E$ and $h_C$, respectively.
\item \label{item: nonconnecteddiagonal4} The quotient maps $C \to C'$ and $E \to E'$ are $N_{\Aut_{E \times C}}(G)$-equivariant.
\end{enumerate}
\end{Lemma}

\proof
Claim \eqref{item: nonconnecteddiagonal1} is clear.

For Claim \eqref{item: nonconnecteddiagonal2}, the inclusion $N_{\Aut_{E \times C}}^\circ(G) \supseteq C_{\Aut_E}^\circ(G) \times C_{\Aut_C}^\circ(\alpha(G))$ follows from Claim \eqref{item: nonconnecteddiagonal1} and we have to show the other inclusion.
By \cite[Corollary 4.2.7]{Brion}, we have $\Aut_{E \times C}^\circ = \Aut_E^\circ \times \Aut_C^\circ$. In particular, being connected, $N_{\Aut_{E \times C}}^\circ(G)$ is contained in $\Aut^\circ_{E} \times \Aut^\circ_C$. 
Hence, it suffices to show that $N_{\Aut_{E \times C}}^\circ(G)$ centralizes $G$ on the first factor of $E \times C$. Since $G \subseteq \Aut_E^\circ$ is a subgroup scheme of the connected commutative group scheme $\Aut_E^\circ$, we have $\Aut_E^\circ \subseteq C^\circ_{\Aut_E}(G) \subseteq N^\circ_{\Aut_E}(G) \subseteq \Aut_E^\circ$, so $N_{\Aut_{E}}^\circ(G)$ centralizes $G$. Therefore, $N_{\Aut_{E \times C}}^\circ(G)$ centralizes $G$ as well.

Claim \eqref{item: nonconnecteddiagonal3} holds for $N_{\Aut_{E \times C}}^{\circ}(G)$ by Claim \eqref{item: nonconnecteddiagonal2}, so it suffices to prove the statement for $T = \Spec k$. Let $h \in N_{\Aut_{E \times C}}(G)(k)$. Since $h$ normalizes $G$, it descends to $X$ by Proposition \ref{prop: normalizer}. The induced automorphism of $X$ preserves both $f_C$ and $f_E$, because they are the only fibrations of $X$ and $E'$ has genus $1$, while $C' \cong \mathbb{P}^1$.
 Since the projections $E \times C \to E$ and $E \times C \to C$ coincide with the Stein factorizations of $f_E \circ \pi$ and $f_C \circ \pi$, respectively, both projections are preserved by $h$. Hence, $h \in \Aut(E) \times \Aut(C)$. An automorphism of this form normalizes the $G$-action on $E \times C$ if and only if it normalizes the $G$-action on both factors and the automorphisms of $G$ induced by the two conjugations are identified via $\alpha$. This proves Claim \eqref{item: nonconnecteddiagonal3}.

Claim \eqref{item: nonconnecteddiagonal4} follows from the $N_{\Aut_{E \times C}}(G)$-equivariance of $\pi,f_E,$ and $f_C$, since the two projections $E \times C \to E$ and $E \times C \to C$ are faithfully flat.
\qed \medskip

Recall that, by Proposition \ref{prop: normalizer}, the action of $N_{\Aut_{E \times C}}(G)$ on $E \times C$ descends to $X$ and we denote the corresponding homomorphism by $\pi_*: N_{\Aut_{E \times C}}(G) \to \Aut_X$. 
After these preparations, we are ready to prove the following refined version of Theorem \ref{Main}.

\begin{Theorem}[\emph{cf.} Theorem \ref{Main}] \label{Mainwithdetails}
Let $X = (E \times C)/G$ be a bielliptic or quasi-bielliptic surface. Then:
\begin{enumerate}
\item \label{item: Main1} $\Aut_{X/C'} = \pi_*(C_{\Aut_E}(G) \times (C_{\Aut_C}(\alpha(G)) \cap \Aut_{C/C'}))$.
\item \label{item: Main2} If $G$ is \'etale, then $\Aut_{X/C'} \cong C_{\Aut_E}(G)$.
\item \label{item: Main3}  $\Aut_{X/E'} = \pi_*((C_{\Aut_E}(G) \cap \Aut_{E/E'}) \times C_{\Aut_C}(\alpha(G)))$.
\item \label{item: Main4} $\Aut_{X/E'} \cong C_{\Aut_C}(\alpha(G))$.
\item \label{item: Main5} There is a short exact sequence of group schemes
$$
1 \to (C_{\Aut_E}(G) \times C_{\Aut_C}(\alpha(G)))/G \overset{\pi_*}{\to} \Aut_X \to M \to 1,
$$
where $G$ is embedded via ${\rm id} \times \alpha$, $M$ is finite and \'etale, and $M(k)$ is a subquotient of the groups $\Aut_{E'}(k)/((f_E)_* C_{\Aut_E}(G)(k))$ and $N_{\Aut(C)}(\alpha(G)(k))/(C_{\Aut_C}(\alpha(G))(k))$.
\item \label{item: Main6} If every element of $M(k)$ can be represented by an automorphism of $X$ that lifts to $E \times C$, then
$$
M(k) \cong \frac{\{(h_E,h_C) \in N_{\Aut_E}(G) \times N_{\Aut_C}(\alpha(G)) \mid \alpha \circ {\rm ad}_{h_E} = {\rm ad}_{h_C} \circ \alpha \}}{ C_{\Aut_E}(G)(k) \times C_{\Aut_C}(\alpha(G))(k)}.
$$
This always holds if $X$ is bielliptic.
\end{enumerate}
\end{Theorem}

\begin{proof}
For Claim \eqref{item: Main1}, we first show that the $\Aut_{X/C'}$-action lifts to $E \times C$. For this, choose a general point $c \in C$ and let $c' \in C'$ be its image in $C'$, so that $\pi$ restricted to $E \times \{c\}$ yields an identification of $E$ with the fiber $F$ of $f_C$ over $c'$. Via this identification, the morphism $(f_E)|_{F}: F \to E'$ is identified with the quotient map $E \to E/G = E'$.
By Lemma \ref{lem: equivariant}, the action of $\Aut_{X/C'}$ on $X$ descends to an action on $E'$, and we can use the restriction homomorphism $\Aut_{X/C'} \to \Aut_F$ and the identification of $F$ with $E$ to get a compatible action of $\Aut_{X/C'}$ on $E$. Using the isomorphism $\pi_E: E \times C \to X \times_{E'} E$, we thus obtain an action of $\Aut_{X/C'}$ on $E \times C$ that lifts the action of $\Aut_{X/C'}$ on $X$. Hence, $\Aut_{X/C'}$ is in the image of $\pi_*$ and it remains to describe its preimage.

By Lemma \ref{lem: nonconnecteddiagonal} \eqref{item: nonconnecteddiagonal4}, a subgroup scheme $H \subseteq N_{\Aut_{E \times C}}(G) \subseteq \Aut_E \times \Aut_C$ maps to $\Aut_{X/C'}$ via $\pi_*$ if and only if it maps to $\Aut_{C/C'}$ under the second projection. To prove Claim \eqref{item: nonconnecteddiagonal1}, we have to show that such an $H$ in fact centralizes $G$. By Lemma \ref{lem: nonconnecteddiagonal} \eqref{item: nonconnecteddiagonal2} this holds for $H^\circ$, so we have to prove that $H(k)$ centralizes $\alpha(G)$. 
Observe that $H(k)$ is mapped to $\Aut_{C/C'}(k)$ under the second projection and $\Aut_{C/C'}(k) = \alpha(G)(k)$ by Proposition \ref{prop: normalizerrationalpoints}. This, and the fact that $G$ is abelian, implies that $H(k)$ centralizes $\alpha(G)$. Now, Lemma \ref{lem: nonconnecteddiagonal} \eqref{item: nonconnecteddiagonal3}, shows that $H(k)$ centralizes the $G$-action on $E \times C$.

For Claim \eqref{item: Main2}, it suffices to show that $C_{\Aut_C}(\alpha(G)) \cap \Aut_{C/C'} = \alpha(G)$, since there is an isomorphism $(C_{\Aut_E}(G) \times \alpha(G))/G \cong C_{\Aut_E}(G)$. This holds if $G$ is \'etale, for then $\Aut_{C/C'}$ is the constant group scheme associated to $\alpha(G)$ by Proposition \ref{prop: normalizerrationalpoints}. 

For Claim \eqref{item: Main3}, we only have to show that the $\Aut_{X/E'}$-action lifts to $E \times C$, because the description of the preimage of $\Aut_{X/E'}$ under $\pi_*$ works as in the proof of Claim \eqref{item: Main1}. Since $\Aut_{X/E'}$ acts trivially on $E'$, we can use the trivial action of $\Aut_{X/E'}$ on $E$ to define an action of $\Aut_{X/E'}$ on $X \times_{E'} E$ lifting the action of $\Aut_{X/E'}$ on $X$. Using the isomorphism $\pi_E: E \times C \to X \times_{E'} E$, we thus obtain the desired lifting.

For Claim \eqref{item: Main4}, we use that the $G$-action on $E$ is free. By Proposition \ref{prop: normalizer} this implies that $\Aut_{E/E'} = G$. Hence, Claim \eqref{item: Main3} shows that $\Aut_{X/E'} = \pi_*(G \times C_{\Aut_C}(\alpha(G)) \cong C_{\Aut_C}(\alpha(G))$. 
 
Next, let us prove Claim \eqref{item: Main5}. By Proposition \ref{prop: normalizer}, the image of $C_{\Aut_E}(G) \times C_{\Aut_C}(\alpha(G))$ under $\pi_*$ is isomorphic to $(C_{\Aut_E}(G) \times C_{\Aut_C}(\alpha(G)))/G$. By Claim \eqref{item: Main1} and Claim \eqref{item: Main3}, this image coincides with the subgroup scheme of $\Aut_X$ generated by the two normal subgroup schemes $\Aut_{X/C'}$ and $\Aut_{X/E'}$, hence it is itself normal. In particular, the quotient $M$ and the exact sequence in Claim \eqref{item: Main3} exist. It remains to describe~$M$.

First, consider the exact sequence
\begin{equation} \label{eq: sequenceinmainproof}
1 \to C_{\Aut_C}(\alpha(G)) \to \Aut_X \overset{(f_E)_*}{\to} \Aut_{E'}
\end{equation}
from Lemma \ref{lem: equivariant}, where we used Claim \eqref{item: Main4} to describe the kernel of $(f_E)_*$. 
The homomorphism $(f_E)_*$ identifies the group scheme $M$ with a subgroup scheme of $\Aut_{E'}/((f_E)_* C_{\Aut_E}(G))$. 
We can choose the image $O' \in E'$ of $O \in E$ as the neutral element of a group law on $E'$. Then, we have $\Aut_{E'} \cong E' \rtimes \Aut_{E',O'}$ for the finite and \'etale stabilizer $\Aut_{E',O'}$ of $O'$. 
Using the translation action, we can consider $E$ as a subgroup scheme of $\Aut_{E} \times \Aut_C$. In fact, we have $E \subseteq C_{\Aut_{E}}(G)$, since $G \subseteq E$ and $E$ is commutative. By Lemma \ref{lem: nonconnecteddiagonal} \eqref{item: nonconnecteddiagonal4}, the induced action on $E'$ coincides with the translation action of $E'$ on itself. Hence, $E' \subseteq (f_E)_* C_{\Aut_E}(G)$. In particular, $M$ is a subquotient of $\Aut_{E',O'}$ and hence it is finite and \'etale.

Let $H := (f_E)_*^{-1} (\Aut_{E',O'}) \subseteq \Aut_X$ and let $F$ be the fiber of $f_E$ over $O'$. Then, the restriction of $\pi$ to $\{O\} \times C$ gives an identification of $C$ with $F$ such that the quotient map $\varphi: C \to C/\alpha(G) = C'$ is identified with $(f_C)|_{F}: F \to C'$. In the following, we use this identification to write $C$ instead of $F$ and $\varphi$ instead of $(f_C)|_F$. Since $\Aut_{E',O'}$ fixes $O'$, the action of $H$ on $X$ preserves $C$ and the morphism $\varphi$ is $H$-equivariant, since $f_C$ is $\Aut_X$-equivariant by Lemma \ref{lem: equivariant}.
In other words, the $H$-action on $C$ factors through $\Aut_{\varphi}$. By Claim \eqref{item: Main1}, the kernel of this action is contained in $(C_{\Aut_E}(G) \times C_{\Aut_C}(\alpha(G)))/G$, hence $M$ is a subquotient of $\Aut_{\varphi}/C_{\Aut_C}(\alpha(G))$. Now, it suffices to observe that $\Aut_{\varphi}(k) = N_{\Aut(C)}(\alpha(G)(k))$, which follows from Proposition \ref{prop: normalizerrationalpoints}.

Finally, for Claim \eqref{item: Main6}, the description of $M$ follows immediately from Lemma \ref{lem: nonconnecteddiagonal} \eqref{item: nonconnecteddiagonal3} and Proposition \ref{prop: normalizer}. By the previous paragraph, we can lift every element of $M(k)$ to an automorphism $g \in \Aut(X)$ mapping under $(f_C)_*$ to the image of $\Aut_{\varphi}(k) \to \Aut_{C'}(k)$. In particular, $g$ lifts to an automorphism $h'$ of $X \times_{C'} C$. Since $\pi_C: E \times C \to E \times_{C'} C$ is birational, we obtain a birational automorphism $h$ of $E \times C$. Now, if $X$ is bielliptic, then $E \times C$ is smooth, minimal, and non-ruled hence $h$ extends to a biregular automorphism of $E \times C$ lifting $g$.
\end{proof}

\begin{Remark} \label{rem: usuallybigger}
We remark that if $G$ is not \'etale, then the group $N_{\Aut(C)}(\alpha(G)(k))$ will usually be bigger than $N_{\Aut_C}(\alpha(G))(k)$. Only later it will turn out that $M$ is in fact a subquotient of the smaller group $N_{\Aut_C}(\alpha(G))(k)/C_{\Aut_C}(\alpha(G))(k)$ in every case.
\end{Remark}

\begin{Remark} \label{rem: M}
In the case-by-case analysis of quasi-bielliptic surfaces in Section \ref{sec: quasi-bielliptic}, we will show that the assumptions of Theorem \ref{Mainwithdetails} \eqref{item: Main6} are also satisfied for all quasi-bielliptic surfaces, hence the description of $M$ also holds for these surfaces.
\end{Remark}

\begin{Remark} \label{rem: mistake}
It will follow from the calculations of Section \ref{sec: normalizercentralizer} that Theorem \ref{Mainwithdetails} \eqref{item: Main2} holds for all bielliptic surfaces. Indeed, the situation where $X$ is bielliptic and $G$ is not \'etale only occurs if $p = 2$ and $G = \mu_2 \times \ZZ/2\ZZ$ and in this case explicit calculations show that $C_{\Aut_C}(\alpha(G)) \cap \Aut_{C/C'} = \alpha(G)$, hence the existence of an isomorphism $\Aut_{X/C'} \cong C_{\Aut_E}(G)$ follows from Theorem \ref{Mainwithdetails} \eqref{item: Main1}.
In particular, for bielliptic surfaces, we always have $\Aut_{X/C'} \cap \Aut_{X/E'} = \pi_*(G \times \alpha(G)) \cong G$.

If $X$ is quasi-bielliptic, then it is not true in general that $\Aut_{X/C'} \cong C_{\Aut_E}(G)$. Indeed, for example if $p = 3$ and $G = \alpha_3$, then $C \to C'$ is purely inseparable of degree $3$, hence $\Aut_{C/C'} = \Aut_C[F]$. Calculations (see Section \ref{quellchar3}, Case (d)) show that $C_{\Aut_C}(\alpha_3)^\circ \cong \alpha_3^2$ and $C_{\Aut_E}(G) \cong E \rtimes \ZZ/3\ZZ$. Hence, by Theorem \ref{Mainwithdetails} \eqref{item: Main1}, $\Aut_{X/C'}$ is non-reduced while $C_{\Aut_E}(G)$ is reduced, so they cannot be isomorphic.
In particular, for quasi-bielliptic surfaces, $\Aut_{X/C'} \cap \Aut_{X/E'}$ can be larger than $G$.
\end{Remark}

We end this section with a description of $(\Aut_X^\circ)_{\red}$. We are thankful to the editors for sharing an observation that allowed us to avoid forward references to Section \ref{sec: normalizercentralizer} in the proof of the following proposition.
\begin{Corollary} \label{cor: autx0}
We have $E \cong (\Aut_X^\circ)_{\red}$ and $(\Aut_X^\circ)_{\red}$ is normal in $\Aut_X$.
\end{Corollary}
\begin{proof}
Since $X$ is not birationally ruled, \cite[Theorem 1]{Popov} implies that $(\Aut_X^\circ)_{\red}$ does not contain a connected linear algebraic group, hence, by \cite[Theorem 1.1.1]{Brion}, $(\Aut_X^\circ)_{\red}$ is an Abelian variety. 
Then, by \cite[Proposition 2.2.1]{Brion}, the stabilizers of the $(\Aut_X^\circ)_{\red}$-action on $X$ are finite. Since $X$ is not an Abelian surface, the $(\Aut_X^\circ)_{\red}$-action on $X$ cannot be transitive, hence $(\Aut_X^\circ)_{\red}$ is either trivial or an elliptic curve. Now, by Theorem \ref{Mainwithdetails}, the action of $E$ on the first factor of $E \times C$ descends to a faithful action of $E$ on $X$. This yields a monomorphism, and hence an isomorphism, of elliptic curves $E \to (\Aut_X^\circ)_{\red}$.
To see that $(\Aut_X^\circ)_{\red}$ is normal in $\Aut_X$, let $(\Aut_{X}^{\circ})_{{\rm ant}}$ be the largest anti-affine subgroup scheme of $\Aut_X^{\circ}$ (see \cite[Chapter 5]{Brion}). By \cite[Lemma 5.1.1]{Brion}, $(\Aut_{X}^{\circ})_{{\rm ant}}$ is smooth and connected, and it contains $(\Aut_X^\circ)_{\red}$, since the latter is anti-affine. Hence, $(\Aut_X^\circ)_{\red} = (\Aut_{X}^{\circ})_{{\rm ant}}$. By \cite[Theorem 1.2.1, Remark 1.2.2]{Brion}, $(\Aut_X^{\circ})_{{\rm ant}} = (\Aut_X^{\circ})_{\red}$ is normal in $\Aut_X^{\circ}$, hence also normal in $\Aut_X$. This finishes the proof.
\end{proof}

%
%

\section{Computing centralizers and normalizers}
\label{sec: normalizercentralizer}
First, recall that if $D$ is an integral curve of arithmetic genus $p_a(D) = 1$ with smooth locus $D^{sm}$ and with a chosen point $O \in D^{sm}$, then there is a decomposition $\Aut_D = D^{sm} \rtimes \Aut_{D,O}$, where the group scheme $\Aut_{D,O}$ of automorphisms fixing $O$ acts on the group scheme $D^{sm}$ of translations $t_s$ by points $s \in D^{sm}$ via $g \circ t_s \circ g^{-1} = t_{g(s)}$. This is because $\Aut_{D,O}$ acts on $D^{sm}$ via group scheme automorphisms, see \cite[Theorem 4.8]{Silverman} and \cite[Proposition 6]{BombieriMumford3}. We use the letter $D$ here, since, using the terminology of Section \ref{sec: main}, the following Lemma \ref{translation} applies to both $D = E$ and $D = C$.


\begin{Lemma}\label{translation}
Let $D$ be an integral curve with $p_a(D) = 1$ and with a chosen point $O \in D^{sm}$. Let $G_1 \subseteq \Aut_{D,O}$ and $G_2 \subseteq D^{sm}$ be subgroup schemes.
Then, the following are equivalent:
\begin{enumerate}
\item \label{item: translation1} $G_2$ normalizes $G_1$.
\item \label{item: translation2} $G_1$ and $G_2$ commute.
\item \label{item: translation3} $G_2 \subseteq D^{G_1}$.
\end{enumerate}
\end{Lemma}

\begin{proof}
Note that if $T$ is a $k$-scheme, $g \in G_1(T)$, and $t_s \in G_2(T)$, then we have
\begin{equation}
\label{eq: translation}
 t_s  \circ g \circ t_{-s} = t_{s - g(s)} \circ g.
\end{equation}
In particular, if $s = g(s)$ for all $t_s \in G_2(T)$, then $G_1$ and $G_2$ commute, hence $\eqref{item: translation3} \Rightarrow \eqref{item: translation2}$. The implication $\eqref{item: translation2} \Rightarrow \eqref{item: translation1}$ is clear, hence it remains to prove $\eqref{item: translation1} \Rightarrow \eqref{item: translation3}$:
if $G_2$ normalizes $G_1$, then Equation \eqref{eq: translation} shows that $t_{s - g(s)}(O_{T'}) = O_{T'}$ for all $T$-schemes $T'$ and $t_s \in G_2(T)$. This is only possible if $s = g(s)$, hence $G_2(T) \subseteq D^{G_1}(T)$.
\end{proof}

\subsection{Bielliptic surfaces}
We use the notation of Section \ref{sec: main} and Lemma \ref{translation}, but assume that $D$ is smooth. In each of the cases $p \neq 2,3$, $p = 3$ and $p = 2$, we will recall the structure of the subgroup scheme $\Aut_{D,O} \subseteq \Aut_D$. Moreover, for every commutative subgroup $H \subseteq \Aut_{D,O}$, we list the fixed locus $D^H$ and, if $\Aut_{D,O}$ is non-commutative, also the centralizer and normalizer of $H$ in Lemma \ref{char0}, Lemma \ref{char3}, and Lemma \ref{char2}. All of this is well-known and elementary to check, and we refer the reader to \cite[Section III.10 and Appendix A]{Silverman} for details. 
Together with Lemma \ref{translation}, it will be straightforward to calculate the groups $C_{\Aut_E}(G)$ and $C_{\Aut_C}(\alpha(G))$ of Theorem \ref{Main} and produce Table \ref{bielliptic}. We will leave the details to the reader, but we will explain how the calculations work in Example \ref{ex: example}. Using Theorem \ref{Mainwithdetails} \eqref{item: Main6}, we calculate $M$ in every case. The results of the calculations of this section are summarized in Table \ref{bielliptic}.
To simplify notation, we define $$N := N_{\Aut(C)}(\alpha(G)(k))/(C_{\Aut_C}(\alpha(G))(k)).$$

\subsubsection{Characteristic $p \neq 2,3$}
By Bombieri and Mumford \cite[p.37]{BombieriMumford2}, the group schemes $G$ leading to bielliptic surfaces $X = (E \times C)/G$ are the seven groups 
$$\ZZ/2\ZZ,\ZZ/3\ZZ,\ZZ/4\ZZ,\ZZ/6\ZZ,(\ZZ/2\ZZ)^2,(\ZZ/3\ZZ)^3,\ZZ/4\ZZ \times \ZZ/2\ZZ.$$ The translation subgroup of $\alpha(G)$ is trivial in the first four of these cases, and isomorphic to the group $\ZZ/2\ZZ,\ZZ/3\ZZ,$ or $\ZZ/2\ZZ$ in the other three cases, respectively.

\begin{Lemma}\label{char0}
The non-trivial commutative subgroup schemes $H$ of $\Aut_{D,O}$ and their fixed loci $D^H$ are as in Table~\ref{Table0}.

\begin{table}[!h]
\centering
$$
\begin{array}{|c|c|c|c|} \hline
j(D) & \Aut_{D,O} &  H & D^H \\ \hline \hline
\neq 0,1728 & \ZZ/2\ZZ & \ZZ/2\ZZ & (\ZZ/2\ZZ)^2 \\ \hline 
1728 & \ZZ/4\ZZ & \begin{array}{c} \ZZ/2\ZZ \\ \ZZ/4\ZZ \end{array} &  \begin{array}{c} (\ZZ/2\ZZ)^2 \\ \ZZ/2\ZZ \end{array} \\ \hline 
0 & \ZZ/6\ZZ & \begin{array}{c} \ZZ/2\ZZ \\ \ZZ/3\ZZ \\ \ZZ/6\ZZ \end{array} &  \begin{array}{c} (\ZZ/2\ZZ)^2 \\ \ZZ/3\ZZ \\ \{1\} \end{array}  \\ \hline
\end{array}
$$
\caption{$\Aut_{D,O}$ and its subgroups in characteristic $\neq 2,3$}
\label{Table0}
\end{table}
\end{Lemma}

\begin{Example}\label{ex: example}
We explain how to calculate the centralizers in the case where the group is $G = \ZZ/2\ZZ$.

For the calculation of $C_{\Aut_E}(G)$, recall that translations in $E$ always commute with $G$. Next, by Lemma \ref{translation}, an automorphism $h_E \in \Aut_{E,O}$ commutes with $G$ precisely if $G \subseteq E^{h_E}$. Now, we apply Lemma \ref{char0}:
if $j(E) \neq 1728$, or $j(E) = 1728$ and $G$ does not coincide with the fixed locus of an automorphism $h_E$ of order $4$ in $\Aut_{E,O}$, then $C_{\Aut_E}(G)/E \cong \ZZ/2\ZZ$. This is Case a) in the first row of Table \ref{bielliptic}. 
If $j(E) = 1728$ and $G = E^{h_E}$, then $C_{\Aut_E}(G)/E \cong \ZZ/4\ZZ$. This is Case b) in the first row of Table \ref{bielliptic} and it seems to be missing from \cite[Table 3.2]{BennettMiranda}, see also Remark \ref{BennettMiranda}.

For the calculation of $C_{\Aut_C}(\alpha(G))$, we apply Lemma \ref{translation} to find the subgroup of translations of $C$ that commute with $\alpha(G)$. By Lemma \ref{char0}, this group is isomorphic to $(\ZZ/2\ZZ)^2$. Next, by Lemma \ref{char0}, the group $\alpha(G)$ is in the center of $\Aut_{C,O}$, so $C_{\Aut_C}(\alpha(G)) \cong (\ZZ/2\ZZ)^2 \rtimes \Aut_{C,O}$. Now, if $j(E) \neq 0,1728$, then $C_{\Aut_C}(\alpha(G))/\alpha(G) \cong (\ZZ/2\ZZ)^2$, if $j(E) = 1728$, then $C_{\Aut_C}(\alpha(G))/\alpha(G) \cong D_8$, and if $j(E) = 0$, then $C_{\Aut_C}(\alpha(G))/\alpha(G) \cong A_4$. These are the Cases i), ii), and iii) in the first row of Table \ref{bielliptic}.
\end{Example}

Similarly, one can calculate the centralizers of $G$ and $\alpha(G)$ for all seven possibilities of $G$. They are listed in Table \ref{bielliptic}. As for the group $N$, we have the following:

\begin{Lemma} \label{lem: N0}
 The group $N$ is as in Table \ref{TableN0}.
\vspace{0.1cm}
\begin{table}[h]
\centering
$$
\begin{array}{|c||c|c|c|c|c|c|c|} \hline
G & \ZZ/2\ZZ  & \ZZ/3\ZZ & \ZZ/4\ZZ & \ZZ/6\ZZ & (\ZZ/2\ZZ)^2 & (\ZZ/3\ZZ)^2  & \ZZ/4\ZZ \times \ZZ/2\ZZ  \\ \hline 
N & \{1\}  & \{1\} & \{1\} & \{1\} & \ZZ/2\ZZ & S_3  &  \ZZ/2\ZZ  \\ \hline
\end{array}
$$
\caption{The group $N$ in characteristic $\neq 2,3$}
\label{TableN0}
\end{table}
\end{Lemma}

\begin{proof}
If $\alpha(G)$ does not contain translations, then $N_{\Aut_C}(\alpha(G)) = C_{\Aut_C}(\alpha(G))$ by Lemma \ref{translation} and because $\Aut_{C,O}$ is abelian. Hence, $N$ is trivial in these cases.

If $G = (\ZZ/2\ZZ)^2$, then conjugation by $N_{\Aut_C}(\alpha(G))$ fixes the unique non-trivial $2$-torsion point $c$ in $\alpha(G)$. By Lemma \ref{char0} and Lemma \ref{translation}, this implies $|N| \mid 2$. The non-trivial element of $N$ is induced by a $4$-torsion point $c'$ of $C$ with $2c' = c$.

If $G = (\ZZ/3\ZZ)^2$, then conjugation by $N_{\Aut_C}(\alpha(G))$ preserves the subgroup $\langle c \rangle \subseteq \alpha(G)$ generated by a non-trivial $3$-torsion point $c$ in $\alpha(G)$. Thus, the action of $N_{\Aut_C}(\alpha(G))$ descends to $C'' := C/\langle c \rangle$. There, it maps to the normalizer in $\Aut_{C''}$ of a subgroup $G'' \subseteq \Aut_{C'',O''}$ of order $3$, where $O''$ is the image of $O$.  By Lemma \ref{translation} and Table \ref{Table0}, the normalizer of $G''$ is isomorphic to $\ZZ/3\ZZ \rtimes \ZZ/6\ZZ$, where $G''$ sits inside the second factor. Thus, $N$ is isomorphic to a subgroup of $S_3$. One can check that the involution in $\Aut_{C,O}$ and a $3$-torsion point not contained in $\langle c \rangle$ induce non-trivial elements of $N$, hence $N \cong S_3$.

Finally, if $G = \ZZ/4\ZZ \times \ZZ/2\ZZ$, then, again, conjugation by $N_{\Aut_C}(\alpha(G))$ fixes the unique non-trivial $2$-torsion point $c$ in $\alpha(G)$. In this case, however, the involution in $\alpha(G) \cap \Aut_{C,O}$ is the unique element in $\alpha(G)$ which is divisible by $2$, hence it is also fixed by $N_{\Aut_C}(\alpha(G))$. Thus, by Lemma \ref{translation}, a translation can be in $N_{\Aut_C}(\alpha(G))$ only if it is a translation by a $2$-torsion point. The non-trivial $2$-torsion point that commutes with $\alpha(G)$ is already contained in $\alpha(G)$, hence $N \cong \ZZ/2\ZZ$ is generated by one of the other two non-trivial $2$-torsion points.
\end{proof}

\begin{Proposition}\label{prop: trivialM}
The cases where $M$ is non-trivial are precisely the following:
\begin{enumerate}
\item \label{item: trivialM1} $G = (\ZZ/2\ZZ)^2$, $j(E) = 1728$, and the fixed points $G^{h_E}$ of the automorphism $h_E$ of order $4$ in $\Aut_{E,O}$ act as translations on $C$. In this case, $M = \ZZ/2\ZZ$.
\item \label{item: trivialM2} $G = (\ZZ/3\ZZ)^2$, $j(E) = 0$, and the fixed points $G^{h_E}$ of the automorphism $h_E$ of order $3$ in $\Aut_{E,O}$ act as translations on $C$. In this case, $M = \ZZ/3\ZZ$.
\end{enumerate}
\end{Proposition}

\begin{proof}
Assume that $M$ is non-trivial.
By Theorem \ref{Mainwithdetails} \eqref{item: Main5} and Table \ref{TableN0}, this can only happen if $G \in \{ (\ZZ/2\ZZ)^2,(\ZZ/3\ZZ)^2, \ZZ/4\ZZ \times \ZZ/2\ZZ \}$.

Assume $G = (\ZZ/2\ZZ)^2$. By Theorem \ref{Mainwithdetails} \eqref{item: Main5} and Table \ref{TableN0}, we have $|M| \mid 2$. If $j(E) \neq 1728$, then $\Aut_{E'}/((f_E)_* C_{\Aut_E}(G))$ has odd order, hence $M = \{1\}$ by Theorem \ref{Mainwithdetails} \eqref{item: Main5}. 
If $j(E) = 1728$, we use Theorem \ref{Mainwithdetails} \eqref{item: Main6}:
by our description of the centralizers and normalizers, both $N_{\Aut_E}(G)/C_{\Aut_E}(G)$ and $N_{\Aut_C}(\alpha(G))/C_{\Aut_C}(\alpha(G))$ are isomorphic to $\ZZ/2\ZZ$ and every non-trivial element of $M(k)$ can be represented by $h = (h_E,h_C)$, where $h_E \in \Aut_{E,O}$ is of order $4$ and $h_C$ is translation by a non-trivial $4$-torsion point such that $h_C^2 \in \alpha(G)$. By Lemma \ref{translation} and Table \ref{Table0}, we have $\alpha \circ {\rm ad}_{h_E} = {\rm ad}_{h_C} \circ \alpha$ if and only if the fixed point of $h_E$ maps via $\alpha$ to the unique translation in $\alpha(G)$. This is Case \eqref{item: trivialM1}.

Next, assume $G = (\ZZ/3\ZZ)^2$. Let $h = (h_E,h_C)$ be an automorphism of $E \times C$ lifting a non-trivial element of $M(k)$. By our description of $C_{\Aut_E}(G)$ and $N$, we may assume that $h_C$ is either the involution in $\Aut_{C,O}$ or translation by a $3$-torsion point $c' \not\in \alpha(G)$, and that $h_E \in \Aut_{E,O}$.
If $h_E$ is an involution, then ${\rm ad}_{h_E}$ fixes only the identity in $G$, while ${\rm ad}_{h_C}$ has more fixed points on $\alpha(G)$. Hence, by Theorem \ref{Mainwithdetails} \eqref{item: Main6}, $h$ does not normalize the $G$-action on $E \times C$ in this case, a contradiction to Proposition \ref{prop: normalizer}.
Thus, we may further assume that $j(E) = 0$ and $h_E$ has order $3$. Then, we may assume that $h_C$ is translation by $c'$. By Lemma \ref{translation} and Table \ref{Table0}, we have $\alpha \circ {\rm ad}_{h_E} = {\rm ad}_{h_C} \circ \alpha$ if and only if the fixed points of $h_E$ on $E$ map to translations in $\alpha(G)$. This is Case \eqref{item: trivialM2}.

Finally, assume $G = \ZZ/4\ZZ \times \ZZ/2\ZZ$. Assume $M$ is non-trivial and, using Theorem \ref{Mainwithdetails} \eqref{item: Main6}, let $h = (h_E,h_C)$ be an automorphism mapping to a non-trivial element in $M(k)$. We may assume that $h_E \in \Aut_{E,O}$ is the involution and $h_C$ is a translation by one of the $2$-torsion points not contained in $\alpha(G)$. Observe that ${\rm ad}_{h_E}$ maps elements of order $4$ in $G$ to their inverses while ${\rm ad}_{h_C}$ maps the automorphism $\sigma$ of order $4$ in $\alpha(G) \cap \Aut_{C,O}$ to $\sigma \circ t_c$, where $c$ is the non-trivial $2$-torsion point in $\alpha(G)$. Hence, we have $\alpha \circ {\rm ad}_{h_E} \neq {\rm ad}_{h_C} \circ \alpha$. This contradiction shows that $M = \{1\}$ in this case. 
\end{proof}

\subsubsection{Characteristic $p = 3$}
By Bombieri and Mumford \cite[p.37]{BombieriMumford2}, the groups $G$ leading to bielliptic surfaces $X = (E \times C)/G$ are the six groups 
$$\ZZ/2\ZZ,\ZZ/3\ZZ,\ZZ/4\ZZ,\ZZ/6\ZZ, (\ZZ/2\ZZ)^2,\ZZ/4\ZZ \times \ZZ/2\ZZ.$$ The translation subgroup of $\alpha(G)$ is trivial in the first four of these cases, and isomorphic to $\ZZ/2\ZZ$ in the other two cases.

\begin{Lemma}\label{char3}
The non-trivial commutative subgroup schemes $H$ of $\Aut_{D,O}$, their fixed loci $D^H$, centralizers $C_{\Aut_{D,O}}(H)$ and normalizers $N_{\Aut_{D,O}}(H)$ are as in Table \ref{Table3}.
\begin{table}[!htb]
\centering
$$
\begin{array}{|c|c|c|c|c|c|} \hline
j(D) & \Aut_{D,O}&  H & D^H & C_{\Aut_{D,O}}(H) & N_{\Aut_{D,O}}(H) \\ \hline \hline
\neq 0 & \ZZ/2\ZZ & \ZZ/2\ZZ & (\ZZ/2\ZZ)^2 & \ZZ/2\ZZ & \ZZ/2\ZZ \\ \hline
0 & \ZZ/3\ZZ \rtimes \ZZ/4\ZZ &  
\begin{array}{c} \ZZ/2\ZZ \\ \ZZ/3\ZZ \\ \ZZ/4\ZZ \\ \ZZ/6\ZZ \end{array} & \begin{array}{c} (\ZZ/2\ZZ)^2 \\ \alpha_3 \\ \ZZ/2\ZZ \\ \{1\} \end{array} & \begin{array}{c} \ZZ/3\ZZ \rtimes \ZZ/4\ZZ \\ \ZZ/6\ZZ \\ \ZZ/4\ZZ \\ \ZZ/6\ZZ   \end{array} & \begin{array}{c} \ZZ/3\ZZ \rtimes \ZZ/4\ZZ \\ \ZZ/3\ZZ \rtimes \ZZ/4\ZZ  \\ \ZZ/4\ZZ \\ \ZZ/3\ZZ \rtimes \ZZ/4\ZZ    \end{array} \\ \hline
\end{array}
$$
\caption{$\Aut_{D,O}$ and its subgroups in characteristic $3$}
\label{Table3}
\end{table}
\end{Lemma}

As in characteristic $\neq 2,3$, it is straightforward to calculate the centralizers of $G$ and $\alpha(G)$ and they are listed in Table \ref{bielliptic}.  

\begin{Lemma}
 The group $N$ is as in Table \ref{TableN3}.

\begin{table}[!h]
\centering
$$
\begin{array}{|c||c|c|c|c|c|c|} \hline
G & \ZZ/2\ZZ & \ZZ/3\ZZ & \ZZ/4\ZZ & \ZZ/6\ZZ  & (\ZZ/2\ZZ)^2 & \ZZ/4\ZZ \times \ZZ/2\ZZ \\ \hline 
N & \{1\} & \ZZ/2\ZZ & \{1\} &  \ZZ/2\ZZ  & \ZZ/2\ZZ & \ZZ/2\ZZ  \\ \hline
\end{array}
$$
\caption{The group $N$ in characteristic $3$}
\label{TableN3}
\end{table}
\end{Lemma}

\begin{proof}
If $\alpha(G)$ does not contain translations, then a translation in $\Aut_C$ normalizes $\alpha(G)$ if and only if it centralizes $\alpha(G)$ by Lemma \ref{translation}. Thus, in these cases, $N$ can be read off from the last two columns of Table \ref{Table3}. The proof of the two remaining cases is the same as for Lemma \ref{lem: N0}.
\end{proof}

\begin{Proposition}\label{prop: Mtrivial3}
The cases where $M$ is non-trivial are precisely the following:
\begin{enumerate}
\item $G = (\ZZ/2\ZZ)^2$ and $j(E) = 0$. In this case, $M \cong \ZZ/2\ZZ$.
\item $G \in \{ \ZZ/3\ZZ, \ZZ/6\ZZ\}$. In these cases, $M \cong \ZZ/2\ZZ$.
\end{enumerate}
\end{Proposition}

\begin{proof}
By Theorem \ref{Mainwithdetails} \eqref{item: Main5} and Table \ref{TableN3}, we may assume $G \in \{\ZZ/3\ZZ,\ZZ/6\ZZ , (\ZZ/2\ZZ)^2, \ZZ/4\ZZ \times \ZZ/2\ZZ\}$.
For $G \in \{(\ZZ/2\ZZ)^2,\ZZ/4\ZZ \times \ZZ/2\ZZ \}$, the proof is essentially the same as in Proposition \ref{prop: trivialM}. The only difference is that every non-trivial $2$-torsion point of $E$ is fixed by some automorphism of order $4$ in $\Aut_{E,O}$, so we do not have an extra condition as in Proposition \ref{prop: trivialM}.

For $G \in \{\ZZ/3\ZZ, \ZZ/6\ZZ\}$, it suffices to find a non-trivial element in $M$. 
By Lemma \ref{char3}, there is an element $h_C \in N_{\Aut_{C,O}}(\alpha(G))$ of order $4$ such that ${\rm ad}_{h_C}$ swaps the two generators of $\alpha(G)$. The inversion $h_E$ on $E$ induces the same action on $G$. By Theorem \ref{Mainwithdetails} \eqref{item: Main6}, this shows $M \cong \ZZ/2\ZZ$.
\end{proof}

\subsubsection{Characteristic $p = 2$}
By Bombieri and Mumford \cite[p.37]{BombieriMumford2}, the group schemes $G$ leading to bielliptic surfaces $X = (E \times C)/G$ are the six group schemes 
$$\ZZ/2\ZZ,\ZZ/3\ZZ,\ZZ/4\ZZ,\ZZ/6\ZZ, \mu_2 \times \ZZ/2\ZZ,(\ZZ/3\ZZ)^2.$$ The translation subgroup scheme of $\alpha(G)$ is trivial in the first four of these cases, and isomorphic to $\mu_2$ and $\ZZ/3\ZZ$, respectively, in the other two cases.

\begin{Lemma}\label{char2}
The non-trivial commutative subgroup schemes $H$ of $\Aut_{D,O}$, their fixed loci $D^H$, centralizers $C_{\Aut_{D,O}}(H)$ and normalizers $N_{\Aut_{D,O}}(H)$ are as in Table \ref{Table2}.
\begin{table}[!htb]
\centering
$$
\begin{array}{|c|c|c|c|c|c|} \hline
j(D) & \Aut_{D,O} &  H & D^H & C_{\Aut_{D,O}}(H) & N_{\Aut_{D,O}}(H) \\ \hline \hline
\neq 0 & \ZZ/2\ZZ & \ZZ/2\ZZ & \mu_2 \times \ZZ/2\ZZ & \ZZ/2\ZZ & \ZZ/2\ZZ \\ \hline
0 & Q_8 \rtimes \ZZ/3\ZZ &  
\begin{array}{c} \ZZ/2\ZZ \\ \ZZ/3\ZZ \\ \ZZ/4\ZZ \\ \ZZ/6\ZZ \end{array} & \begin{array}{c} M_2 \\ \ZZ/3\ZZ \\ \alpha_2 \\ \{1\} \end{array} & \begin{array}{c} Q_8 \rtimes \ZZ/3\ZZ \\ \ZZ/6\ZZ \\ \ZZ/4\ZZ \\ \ZZ/6\ZZ   \end{array} & \begin{array}{c} Q_8 \rtimes \ZZ/3\ZZ \\ \ZZ/6\ZZ  \\ Q_8 \\ \ZZ/6\ZZ   \end{array} \\ \hline
\end{array}
$$
\caption{$\Aut_{D,O}$ and its subgroups in characteristic $2$}
\label{Table2}
\end{table}
\end{Lemma}

As before, it is straightforward to calculate the centralizers of $G$ and $\alpha(G)$ and they are listed in Table \ref{bielliptic}.  

\begin{Lemma}
 The group $N$ is as in Table \ref{TableN2}.

\begin{table}[!h]
\centering
$$
\begin{array}{|c||c|c|c|c|c|c|} \hline
G & \ZZ/2\ZZ  & \ZZ/3\ZZ & \ZZ/4\ZZ & \ZZ/6\ZZ & \mu_2 \times \ZZ/2\ZZ & (\ZZ/3\ZZ)^2   \\ \hline 
N & \{1\} & \{1\} & \ZZ/2\ZZ &  \{1\} & \{1\} & S_3  \\ \hline
\end{array}
$$
\caption{The group $N$ in characteristic $2$}
\label{TableN2}
\end{table}
\end{Lemma}

\begin{proof}
If $\alpha(G)$ does not contain translations, then a translation in $\Aut_C$ normalizes $\alpha(G)$ if and only if it centralizes $\alpha(G)$ by Lemma \ref{translation}. Thus, in these cases, $N$ can be read off from the last two columns of Table \ref{Table2}. 
For $G = (\ZZ/3\ZZ)^2$, the proof is the same as for Lemma \ref{lem: N0}.
Finally, if $G = \mu_2 \times \ZZ/2\ZZ$, then $N_{\Aut(C)}(\alpha(G)(k))$ is generated by $\alpha(G)(k)$ and the unique non-trivial $2$-torsion point in $C(k)$ by the same argument as in the proof of Lemma \ref{translation}. Translation by this $2$-torsion point commutes with $\alpha(G)$, hence $N$ is trivial.
\end{proof}

\begin{Proposition}\label{prop: Mtrivial2}
The cases where $M$ is non-trivial are precisely the following:
\begin{enumerate}
\item $G = (\ZZ/3\ZZ)^2$ and $j(E) = 0$. In this case, $M = \ZZ/3\ZZ$.
\item $G = \ZZ/4\ZZ$. In this case, $M = \ZZ/2\ZZ$.
\end{enumerate}
\end{Proposition}

\begin{proof}
By Theorem \ref{Mainwithdetails} \eqref{item: Main5} and Table \ref{TableN2}, we may assume $G \in \{(\ZZ/3\ZZ)^2,\ZZ/4\ZZ\}$.
The proof for $G =(\ZZ/3\ZZ)^2$ is the same as in Proposition \ref{prop: trivialM} with the only difference that every non-trivial $3$-torsion point in $E$ is fixed by some automorphism of order $3$, so we do not have an extra condition as in Proposition \ref{prop: trivialM}.

If $G = \ZZ/4\ZZ$, consider the automorphism $h = (h_E,h_C)$ of $E \times C$ where $h_C \in N_{\Aut_{C,O}}(\alpha(G))$ is of order $4$ and not contained in $\alpha(G)$ and $h_E$ is the inversion involution on $E$. 
By Lemma \ref{lem: nonconnecteddiagonal} \eqref{item: nonconnecteddiagonal3}, $h$ normalizes the $G$-action on $E \times C$ and, by Proposition \ref{prop: normalizer}, induces a non-trivial element of $M$. Hence, we have $M \cong \ZZ/2\ZZ$.
\end{proof}
 


\subsection{Quasi-bielliptic surfaces}
\label{sec: quasi-bielliptic}
In the case of quasi-bielliptic surfaces, $E$ is still smooth, so the group $C_{\Aut_E}(G)/E$ can be calculated using the results of the previous section. We will thus focus on the calculation of $C_{\Aut_C}(\alpha(G))/\alpha(G)$ and $M$. We identify the smooth locus of $C$ with $\mathbb{A}^1 = \Spec k[t]$ and use the description of automorphisms of $\mathbb{A}^1$ coming from $C$ given in \cite[Proposition 6]{BombieriMumford3}.

\subsubsection{Characteristic $p = 3$}\label{quellchar3}

By \cite[Proposition 6]{BombieriMumford3} the $T$-valued automorphisms of $\mathbb{A}^1$ coming from $C$ are of the form
\begin{equation}\label{quell3eqn}
t \mapsto bt + c + dt^3
\end{equation}
with $b \in \GG_m(T)$, $c,d \in \GG_a(T)$ and $d^3 = 0$. By \cite[p. 214]{BombieriMumford3}, the subgroup schemes $\alpha(G)$ leading to quasi-bielliptic surfaces are the following:
\begin{enumerate}[label=(\alph*)]
\item $\mu_3$: $t \mapsto at + (1-a)t^3$ with $a^3 = 1$
\item $\mu_3 \times \ZZ/2\ZZ: \mu_3$ as in (a) and $t \mapsto \pm t$.
\item $\mu_3 \times \ZZ/3\ZZ: \mu_3$ as in (a) and $t \mapsto t + i$ with $i^3 = i$
\item $\alpha_3: t \mapsto t + at^3$ with $a^3 = 0$
\item $\alpha_3 \times \ZZ/2\ZZ: \alpha_3$ as in (d) and $t \mapsto \pm t$
\end{enumerate}

\begin{Remark} \label{rem: char3fall}
As noted in \cite[p.489]{Langquell}, Case (f) of \cite[p. 214]{BombieriMumford3} does not exist, because the group scheme given there is isomorphic to $\alpha_9$ and thus not a subscheme of an elliptic curve.
\end{Remark}
 
Now, let us calculate $C_{\Aut_C}(\alpha(G))$ and $M$ for the surfaces in Case (a),...,(e). To this end, we take a $k$-scheme $T$ and arbitrary elements $g \in \alpha(G)(T)$ as in the above list and $h \in \Aut_C(T)$ as in \eqref{quell3eqn}. One can check that the inverse of $h$ is given by
$$
t \mapsto b^{-1}t + b^{-4}(c^3d - b^3c) - b^{-4}dt^3
$$

\begin{enumerate}[label=(\alph*),leftmargin=7mm]
\item We calculate
$$
h \circ g \circ h^{-1}: t \mapsto at + (1-a)b^{-1}(c^3-c) + (1-a)(b^2 - b^{-1}d)t^3.
$$
Thus, $h$ normalizes $\alpha(G)$ if and only if it centralizes $\alpha(G)$ if and only if $c^3 = c$ and $b^3 = d + b$. Taking the cube of the second equation, we obtain $b^6 = 1$. Thus, the centralizer of $\alpha(G)$ is the group scheme of maps
$$
t \mapsto bt + i + (b^3 - b)t^3 \text{ with } b^6 = 1 \text{ and } i^3 = i.
$$
This group scheme is isomorphic to $\mu_3 \times S_3$. Therefore, we have $C_{\Aut_C}(\alpha(G))/\alpha(G) \cong S_3$. 

To calculate $M$, first note that $|M| \mid 2$, since $E$ and $E'$ are ordinary, $M$ is a subquotient of $\Aut_{E'}/((f_E)_* C_{\Aut_E}(G))$ by Theorem \ref{Mainwithdetails} \eqref{item: Main5}, and $\Aut_{E'}^\circ \subseteq ((f_E)_* C_{\Aut_E}(G))$. If $M$ is non-trivial, then it can be represented by an automorphism $g \in \Aut(X)$ that induces the inversion involution on $E'$. This involution can be lifted to $E$, hence $g$ lifts to an automorphism of $E \times C$. However, by the above calculations there is no element of $\Aut(C)$ that acts as an inversion on $\alpha(G)$. So, Theorem \ref{Mainwithdetails} \eqref{item: Main6} shows that $M$ is trivial.
\vspace{2mm}

\item
Since $\mu_3$ is the identity component of $\mu_3 \times \ZZ/2\ZZ$, the normalizer of $\mu_3 \times \ZZ/2\ZZ$ in $\Aut_C$ is contained in the normalizer of $\mu_3$ in $\Aut_C$. By Case (a), the latter is isomorphic to $\mu_3 \times S_3$. Thus, $N_{\Aut_C}(\alpha(G))$ equals the normalizer of $\mu_3 \times \ZZ/2\ZZ$ in $\mu_3 \times S_3$, hence $N_{\Aut_C}(\alpha(G)) = C_{\Aut_C}(\alpha(G)) = \alpha(G)$. By the same argument as in (a), we also have $M = \{1\}$.
\vspace{2mm}

\item
Similar to case (b), we obtain that $C_{\Aut_C}(\alpha(G))/\alpha(G) = \{1\}$ and $M = \{1\}$.
\vspace{2mm}

\item We calculate
$$
h \circ g \circ h^{-1}: t \mapsto t + ab^{-1}c^3 + ab^2t^3.
$$
Thus, $h$ normalizes $\alpha(G)$ if and only if $c^3 = 0$, and it centralizes $\alpha(G)$ if and only if additionally \mbox{$b^2 = 1$} holds. Thus, $C_{\Aut_C}(\alpha(G))$ is a semi-direct product $\alpha_3^2 \rtimes \ZZ/2\ZZ$ and $N_{\Aut_C}(\alpha(G))$ is a semi-direct product $(\alpha_3)^2 \rtimes \GG_m$.
In particular, we have $C_{\Aut_C}(\alpha(G))/\alpha(G) \cong \alpha_3 \rtimes \ZZ/2\ZZ$.

Next, we calculate $M$. Using  Lemma \ref{translation} and Lemma \ref{char3}, one can check that $C_{\Aut_E}(G)/E \cong \ZZ/3\ZZ$.
Thus, there is an isomorphism $\Aut_{E'}/((f_E)_* C_{\Aut_E}(G)) \cong \Aut_{E',O}/(\ZZ/3\ZZ) \cong \ZZ/4\ZZ$, where we use the structure of $\Aut_{E',O}$ recalled in Lemma \ref{char3}. So, by Theorem \ref{Mainwithdetails} \eqref{item: Main5}, $M$ is a subquotient of $\ZZ/4\ZZ$.

Choose any automorphism $h_E \in \Aut_{E,O}$ of order $4$. Since $\alpha_3 \subseteq E$ is the kernel of Frobenius, it is preserved by $h_E$. Moreover, by Lemma \ref{translation} and Lemma \ref{char3}, the centralizer of $\alpha_3$ in $\Aut_{E,O}$ has order $3$, so conjugation by $h_E$ induces an automorphism of $\alpha_3$ of order $4$. By the calculations of the first paragraph, we have a surjection $N_{\Aut_C}(\alpha(G)) \to \Aut_{\alpha(G)} \cong \mathbb{G}_m$, hence we can find an $h_C \in N_{\Aut_C}(\alpha(G))(k)$ such that 
$h = (h_E,h_C) \in N_{\Aut_E \times \Aut_C}(G)(k)$ by Lemma \ref{lem: nonconnecteddiagonal} \eqref{item: nonconnecteddiagonal3}. By Proposition \ref{prop: normalizer}, $h$ descends to an automorphism of $X$ that induces an element of order $4$ in $M$. Therefore, we have $M \cong \ZZ/4\ZZ$.
\vspace{2mm}

\item Let $g: t \mapsto - t$. Then,
$$
h \circ g \circ h^{-1}: t \mapsto -t + b^{-1}c - b^{-4}c^3d.
$$
Since $\alpha_3$ is the identity component of $\alpha_3 \times \ZZ/2\ZZ$, we can use the results of (d) to deduce that $h$ normalizes $\alpha(G)$ if and only if $c = 0$ and it centralizes $\alpha(G)$ if and only if additionally $b^2 = 1$. Thus, we get $C_{\Aut_C}(\alpha(G)) \cong \alpha_3 \times \ZZ/2\ZZ$ and the normalizer of $\alpha(G)$ is $N_{\Aut_C}(\alpha(G)) \cong \alpha_3 \rtimes \GG_m$. In particular, $C_{\Aut_C}(\alpha(G))/\alpha(G) = \{1\}$.

Since the automorphism $g$ generates the group $\alpha(G)(k)$, the calculation of the previous paragraph also shows that $N_{\Aut(C)}(\alpha(G)(k)) = \GG_m(k)$. Thus, $M$ is isomorphic to a subquotient of $\GG_m(k)$ by Theorem \ref{Mainwithdetails} \eqref{item: Main5} and, in particular, the order of $M$ is prime to $3$. By the same theorem, $M$ is also a subquotient of $\Aut_{E'}/((f_E)_* C_{\Aut_E}(G))$, which is isomorphic to $\ZZ/3\ZZ \rtimes \ZZ/4\ZZ$ since $C_{\Aut_E}(G) \cong E$ in the current case. Hence, $M$ is a subquotient of $\ZZ/4\ZZ$. Using the same construction as in (d), one can show that $M \cong \ZZ/4\ZZ$.
\vspace{2mm}
\end{enumerate}

\subsubsection{Characteristic $p = 2$} \label{quellchar2}
By \cite[Proposition 6]{BombieriMumford3} the $T$-valued automorphisms of $\mathbb{A}^1$ coming from $C$ are of the form
\begin{equation}\label{quell2eqn}
t \mapsto bt + c + dt^2 + et^4
\end{equation}
with $b \in \GG_m(T)$, $c,d,e \in \GG_a(T)$ and $d^4 = e^2 = 0$. The subgroup schemes $\alpha(G)$ leading to quasi-bielliptic surfaces are the following, where $\lambda \in k$:
\begin{enumerate}[label=(\alph*)]
\item $\mu_2$: $t \mapsto at + \lambda(a+1)t^2 + (a+1)t^4$ with $a^2 = 1$.
\item $\mu_2 \times \ZZ/3\ZZ: \mu_2$ as in (a) with $\lambda = 0$ and $t \mapsto \omega t$, where $\omega^3 = 1$.
\item $\mu_2 \times \ZZ/2\ZZ: \mu_2$ as in (a) and $t \mapsto t + \zeta$, where $\zeta$ is a fixed root of $x^3 + \lambda x + 1$.
\item $\mu_4: t \mapsto at + (a+a^2)t^2 + (1+a^2)t^4$ with $a^4 = 1$
\item $\mu_4 \times \ZZ/2\ZZ: \mu_4$ as in (d) and $t \mapsto t + 1$.
\item $\alpha_2: t \mapsto t + \lambda a t^2 + at^4$ with $a^2 = 0$, and with $\lambda \in \{0,1\}$.
\item $\alpha_2 \times \ZZ/3\ZZ: \alpha_2$ as in (f) with $\lambda = 0$ and $\ZZ/3\ZZ$ as in (b)
\item $M_2: t \mapsto t + a + \lambda a^2 t^2 + a^2 t^4$ with $a^4  =0$, and with $\lambda \neq 0$.
\end{enumerate}

\begin{Remark} \label{rem: char2lambda}
In \cite[p. 214]{BombieriMumford3}, Bombieri and Mumford do not give restrictions on the parameter $\lambda \in k$ in Case (f). However, all the $\alpha_2$-actions with $\lambda \neq 0$ described by them are conjugate, so we may assume $\lambda \in \{0,1\}$. For more details, we refer the reader to the discussion of Case (f) below.
\end{Remark}

\begin{Remark}\label{rem: M2}
To see that the group scheme in Case (h) is indeed $M_2$, denote the transformation in Case (h) associated to $z_i$ with $z_i^4 = 0$ by $t_{z_i}$. Observe that $t_{z_1} \circ t_{z_2} = t_{z_1 + z_2 + \lambda z_1^2 z_2^2}$. So, if $G = \Spec k[z]/z^4$ is the group scheme in Case (h), then its co-multiplication is given by $$z \mapsto z_1 \otimes 1 + 1 \otimes z_2 + \lambda z_1^2 \otimes z_2^2.$$
Consider the supersingular elliptic curve $E$ with affine Weierstrass equation $y^2 + \lambda y = x^3$ and set $z = x/y,w = 1/y$, so that the equation becomes $z^3 = w + \lambda w^2$. Then, the $2$-torsion subscheme $M_2$ of $E$ is the subscheme given by $z^4 = w^2 = 0$, and thus $w = z^3$. By \cite[p.120]{Silverman} the co-multiplication on $k[z]/z^4$ induced by the group structure on $E$ is precisely the one described above. Hence, we have $G = M_2$.

For later use, we note that by \cite[Appendix A, Proposition 1.2]{Silverman}, the group of automorphisms of $E$ preserving $w = z = 0$ is given by the substitutions $x \mapsto b^2x + c^2, y \mapsto y + b^2cx +d $ with $b^3 = 1, c^4 + \lambda c = 0$ and $d^2 + \lambda d + c^6 = 0$. In particular, they act on $k[z]/z^4$ as
$$
z \mapsto \frac{b^2x + c^2}{ y + b^2cx +d} = \frac{b^2z + c^2w}{1+ b^2cz + dw} = (b^2z + c^2z^3)(1 + b^2cz + dz^3)^3 = b^2z + bcz^2.
$$

In particular, if we think of the substitutions in Case (h) above as defining a homomorphism of group schemes $E \supseteq M_2 \overset{\alpha}{\to} \Aut_C$, then precomposing $\alpha$ with ${\rm ad}_{h_E}$ where $h_E \in \Aut_{E,O}$ is as described in the previous paragraph, then $\alpha \circ {\rm ad}_{h_E}$ corresponds to $M_2$ acting on $C$ as 
$$
t \mapsto t + (b^2a + bca^2) + b \lambda a^2t^2 + ba^2t^4.
$$
\end{Remark}
\vspace{5mm}
 
Now, we are prepared to calculate $C_{\Aut_C}(\alpha_C)$ and $M$ in Cases (a),...,(h). As in characteristic $3$, we take a $k$-scheme $T$ and arbitrary elements $g \in \alpha(G)(T)$ as in the above list and $h \in \Aut_C(T)$ as in \eqref{quell2eqn}. One can check that the inverse of $h$ is given by
$$
t \mapsto b^{-1}t + b^{-7}(b^6c + b^2c^4e + b^4c^2d + c^4d^3) + b^{-3} d t^2 + b^{-7}(d^3 + b^2e)t^4.
$$

\begin{enumerate}[label=(\alph*),leftmargin=7mm]
\item We calculate
\begin{eqnarray*}
h \circ g \circ h^{-1}: & t \mapsto  & at + (a+1)b^{-1}(c + \lambda c^2 + c^4)  + (a+1)(b^{-1}d + \lambda b)t^2 \\ & & + (a+1)(b^{-1}e + \lambda b^{-1} d^2 + b^3)t^4.
\end{eqnarray*}
Thus, $h$ normalizes $\alpha(G)$ if and only if it centralizes $\alpha(G)$ if and only if
\begin{eqnarray}
c^4 + \lambda c^2 + c &=& 0, \nonumber \\
d &=& \lambda (b^2 + b), \text{ and} \label{eq: secondequation} \\
e &=& b^4 + b + \lambda^3 (b^4 + b^2) \label{eq: thirdequation}.
\end{eqnarray}

If $\lambda \neq 0$, the fourth power of \eqref{eq: secondequation} yields $b^4 = 1$, while the square of \eqref{eq: thirdequation} yields $b^6 = 1$, so we have $b^2 = 1$. Hence, in this case $C_{\Aut_C}(\alpha(G))$ is the group scheme of maps
$$
t \mapsto bt + c + \lambda(1 + b)t^2 + (1+b)t^4 \text{ with } b^2 = 1 \text{ and } c^4 + \lambda c^2 + c = 0 ,
$$
which is isomorphic to $(\ZZ/2\ZZ)^2 \times \mu_2$ since $c^4 + \lambda c^2 + c$ has $4$ distinct roots. Therefore, we have $C_{\Aut_C}(\alpha(G))/\alpha(G) \cong (\ZZ/2\ZZ)^2$.

If $\lambda = 0$, then $d = 0$, and the square of \eqref{eq: thirdequation} yields $b^6 = 1$. Thus, the centralizer of $\alpha(G)$ is the group scheme of maps
$$
t \mapsto bt + c + (b + b^4)t^4 \text{ with }  b^6 = 1 \text{ and } c^4 = c,
$$
which is isomorphic to $A_4 \times \mu_2$. We deduce that $C_{\Aut_C}(\alpha(G))/\alpha(G) \cong A_4$.

In both cases $\lambda \neq 0$ and $\lambda = 0$, note that $\Aut_{E'} = (f_E)_* C_{\Aut_E}(G)$, so $M = \{1\}$ follows immediately from Theorem \ref{Mainwithdetails} \eqref{item: Main5}.

\vspace{2mm}

\item
Since $\mu_2$ is the identity component of the group scheme $\ZZ/3\ZZ \times \mu_2$, it suffices to calculate the normalizer of $\ZZ/3\ZZ \times \mu_2$ in $A_4 \times \mu_2$, which is equal to its centralizer and both are equal to $\ZZ/3\ZZ \times \mu_2$. In particular, $C_{\Aut_C}(\alpha(G))/\alpha(G) = \{1\}$. To see that $M = \{1\}$, one can use the same arguments as in Case (a) in characteristic $3$ to show that the action of $M$ lifts to $E \times C$. Since $N_{\Aut_C}(\alpha(G)) = C_{\Aut_C}(\alpha(G))$, Theorem~\ref{Mainwithdetails}~\eqref{item: Main6} shows that $M$ is trivial.

\vspace{2mm}

\item
We take the centralizer of $\ZZ/2\ZZ \times \mu_2$ in $(\ZZ/2\ZZ)^2 \times \mu_2$ if $\lambda \neq 0$ and in $A_4 \times \mu_2$ if $\lambda = 0$. Both are equal to the normalizer and also equal to $(\ZZ/2\ZZ)^2 \times \mu_2$. Thus, $C_{\Aut_C}(\alpha(G))/\alpha(G) \cong \ZZ/2\ZZ$. As in Case (a), we have $M = \{1\}$.

\vspace{2mm}

\item We calculate
\begin{eqnarray*}\label{complicatedeqn}
h \circ g \circ h^{-1}: t & \longmapsto &  at + (a+1)b^{-1}\left(c + ac^2 + (a+1)(b^{-2}c^2d + b^{-2}c^4d + c^4)\right)  + (a+a^2)(b^{-1}d + b)t^2 \\
&  &+ (a+1)\left(b^{-1}e + a b^{-1}d^2 + (a+1)(b^{-3}d^3 + bd + b^3)\right)t^4.
\end{eqnarray*}
Thus, $h$ normalizes $\alpha(G)$ if and only if it centralizes $\alpha(G)$. For $h$ to centralize the subgroup scheme where $a^2 = 1$, we obtain the conditions
\begin{eqnarray*}
c + c^2 &=& 0, \\
d &=& b^2 + b, \text{ and} \\
e &=& b^4 + b^2.
\end{eqnarray*}
Since $d^4  = 0$, this implies $b^4  = 1$. Plugging these conditions back into the equation for $h \circ g \circ h^{-1}$, it turns out that the subgroup scheme of transformations satisfying these conditions centralizes all of $\alpha(G)$.
Therefore, the centralizer $C_{\Aut_C}(\alpha(G))$ is given by the group scheme of maps
$$
t \mapsto bt + c + (b + b^2)t^2 + (1 + b^2)t^4 \text{ with } b^4 = 1 \text{ and } c \in \{0,1\},
$$
which is isomorphic to $\mu_4 \times \ZZ/2\ZZ$. Therefore, $C_{\Aut_C}(\alpha(G))/\alpha(G) \cong \ZZ/2\ZZ$. By the same argument as in Case (b), we have $M = \{1\}$.

\vspace{2mm}

\item 
Since $\mu_4$ is the identity component of $\mu_4 \times \ZZ/2\ZZ$, we can use the computations of (d) to immediately conclude that centralizer and normalizer of $\alpha(G)$ are both equal to $\mu_4 \times \ZZ/2\ZZ$ and thus $C_{\Aut_C}(\alpha(G))/\alpha(G)= \{1\}$. Also, $M = \{1\}$ follows by the same argument as in Case (b).

\vspace{2mm}

\item We calculate
\begin{eqnarray*}
h \circ g \circ h^{-1}: t \mapsto t + ab^{-1}(\lambda c^2 + c^4) + \lambda ab t^2 + a (\lambda b^{-1}d^2 + b^3)t^4.
\end{eqnarray*}
This shows that all the $\alpha_2$-actions with $\lambda \neq 0$ are conjugate to the one with $\lambda = 1$ by conjugating with the map $t \mapsto \sqrt{\lambda}t$. Hence, we may assume $\lambda \in \{0,1\}$.

Suppose $\lambda = 1$. Then, $h$ normalizes $\alpha(G)$ if and only if it satisfies the conditions
\begin{eqnarray}
c^2 + c^4 &=& 0, \text{ and} \nonumber \\
d^2 &=& b^4 + b^2 \label{eq: secondequationalpha}.
\end{eqnarray}
Squaring \eqref{eq: secondequationalpha}, we get $b^4 = 1$. We also note that $h$ centralizes $\alpha(G)$ if and only if additionally $b = 1$. Therefore, the normalizer of $\alpha(G)$ is the group scheme of maps
$$
t \mapsto bt + c + dt^2 + et^4 \text{ with } b^4 = 1, c^4 = c^2, d^2 = b^4 + b^2, \text{ and } e^2 = 0,
$$
which is isomorphic to a semi-direct product $(\alpha_2^3 \rtimes \ZZ/2\ZZ) \rtimes \mu_4$ and the centralizer of $\alpha(G)$ is isomorphic to a semi-direct product $\alpha_2^3 \rtimes \ZZ/2\ZZ$. Hence, we have $C_{\Aut_C}(\alpha(G))/\alpha(G) \cong \alpha_2^2 \rtimes \ZZ/2\ZZ$. To calculate $M$, note that $\Aut_{E'}/((f_E)_* C_{\Aut_E}(G)) \cong \ZZ/3\ZZ$, so $|M| \mid 3$ by Theorem \ref{Mainwithdetails} \eqref{item: Main5}. Since $E \to E'$ is purely inseparable we can lift the $M$-action from $E'$ to $E$ and hence to $E \times C$, where it normalizes $G$. Since $N_{\Aut_C}(\alpha(G))/C_{\Aut_C}(\alpha(G))(k) \cong \mu_4(k)$ is trivial, this shows that $M = \{1\}$.

If $\lambda = 0$, then $h$ normalizes $\alpha(G)$ if and only if $c^4 = 0$ and it centralizes $\alpha(G)$ if and only if additionally $b^3 = 1$. Thus, the normalizer of $\alpha(G)$ is the group scheme of maps
$$
t \mapsto bt + c + dt^2 + et^4 \text{ with } c^4 = d^4 = e^2 = 0,
$$
which is isomorphic to $(\alpha_4 \rtimes A) \rtimes \GG_m$ and the centralizer is isomorphic to $(\alpha_4 \rtimes A) \rtimes \ZZ/3\ZZ$, where $A$ is a non-split extension of $\alpha_4$ by $\alpha_2 = \alpha(G)$. Thus, we have $C_{\Aut_C}(\alpha(G))/\alpha(G) \cong (\alpha_4 \rtimes \alpha_4) \rtimes \ZZ/3\ZZ$.

Finally, let us explain how to compute $M$ in the case $\lambda = 0$. As in the case $\lambda = 1$, we have $|M| \mid 3$. Choose an element $h_E \in \Aut_{E,O}$ of order $3$. Since $\alpha(G) = \alpha_2$ is the kernel of Frobenius on $E$, it is preserved by $h_E$. By Lemma \ref{translation} and Lemma \ref{char2}, conjugation by $h_E$ induces an automorphism of $\alpha_2$ of order $3$. On the other hand, the conjugation action of $N_{\Aut_C}(\alpha(G))$ on $\alpha_2$ factors through $N_{\Aut_C}(\alpha(G))/C_{\Aut_C}(\alpha(G))$. By the calculations of the previous paragraph and since $\Aut_{\alpha_2} \cong \GG_m$, we can find an automorphism $h_C \in N_{\Aut_C}(\alpha(G))$ of order $9$ such
that $\alpha \circ {\rm ad}_{h_E} = {\rm ad}_{h_C} \circ \alpha$. By Lemma \ref{lem: nonconnecteddiagonal} \eqref{item: nonconnecteddiagonal3}, $h = (h_E,h_C)$ normalizes the $G$-action on $E \times C$. By Proposition \ref{prop: normalizer}, $h$ descends to $X$ and induces a non-trivial element of $M$. Hence, we have $M \cong \ZZ/3\ZZ$. 

\vspace{2mm}

\item Let $g: t \mapsto \omega t$, where $\omega^2 + \omega = 1$. Then,
$$
h \circ g \circ h^{-1}: t \mapsto \omega t + \omega^2b^{-1}(c + b^{-4}c^4 e + \omega^2 b^{-2}c^2d + b^{-6}c^4 d^3) + b^{-1}dt^2 +  b^{-3}d^3t^4.
$$
Thus, $h$ normalizes $\ZZ/3\ZZ$ if and only if it centralizes $\ZZ/3\ZZ$ if and only if
\begin{eqnarray*}
d &=& 0, \text{ and} \\
c^4e + b^4 c &=& 0.
\end{eqnarray*}
Putting  this together with the conditions obtained in (f), we deduce that the normalizer of $\alpha(G)$ is the group scheme of maps
$$
t \mapsto bt + et^4 \text{ with } e^2 = 0,
$$
which is isomorphic to $\alpha_2 \rtimes \GG_m$. Moreover, we see that $C_{\Aut_C}(\alpha(G)) = \alpha(G)$. 

Since the automorphism $g$ generates the group $\alpha(G)(k)$, the calculation of the previous paragraph also shows that $N_{\Aut(C)}(\alpha(G)(k)) = \GG_m(k)$. Thus, $M$ is a subquotient of $\GG_m(k)$ by Theorem \ref{Mainwithdetails} \eqref{item: Main5} and, in particular, the order of $M$ is prime to $2$. By the same theorem, $M$ is also a subquotient of $\Aut_{E'}/((f_E)_* C_{\Aut_E}(G))$, which is isomorphic to $Q_8 \rtimes \ZZ/3\ZZ$ since $C_{\Aut_E}(G) \cong E$ in the current case. Hence, $M$ is a subquotient of $\ZZ/3\ZZ$. Using the same construction as in (f), one can show that $M = \ZZ/3\ZZ$.

\vspace{2mm}

\item We compute
$$
h \circ g \circ h^{-1}: t \mapsto t + ab^{-1}(1 + a(\lambda c^2 + c^4 + b^{-2}d)) + \lambda a^2 b t^2 + a^2(\lambda b^{-1}d^2 + b^3)t^4.
$$
This means that $h$ normalizes $\alpha(G)$ if and only if it satisfies
\begin{eqnarray}
b^3 &=& 1, \text{ and } \nonumber \\
\lambda d^2 &=& b + b^2. \label{eq: almostfinalequation}
\end{eqnarray}
In fact, since $d^4 = 0$, we can square \eqref{eq: almostfinalequation} to deduce $b = 1$, and since $\lambda \neq 0$, we get $d^2 = 0$. Now, $h$  centralizes $\alpha(G)$ if and only if additionally
\begin{eqnarray}
d &=& c^4 + \lambda c^2. \label{eq: finalequation}
\end{eqnarray}
Squaring \eqref{eq: finalequation}, we obtain $c^8 + \lambda^2 c^4 = 0$. Hence, the centralizer $C_{\Aut_C}(\alpha(G))$ of $\alpha(G)$ is the group scheme of maps
$$
t \mapsto t + c + (c^4 + \lambda c^2)t^2 + et^4 \text{ with } e^2 = 0 \text{ and } c^8 + \lambda^2 c^4 = 0,
$$
which is isomorphic to $(M_2 \times \alpha_2) \rtimes \ZZ/2\ZZ$, and the normalizer of $\alpha(G)$ is the group scheme of maps
$$
t \mapsto t + c + dt^2 + et^4 \text{ with } d^2 = e^2 = 0,
$$
which is isomorphic to $\GG_a \rtimes \alpha_2^2$. In particular, we have $C_{\Aut_C}(\alpha(G))/\alpha(G) \cong \alpha_2 \times \ZZ/2\ZZ$.

To calculate $M$, note first that $M$ is a subquotient of $\Aut_{E'}/((f_E)_*C_{\Aut_E}(G)) \cong A_4$ by Theorem \ref{Mainwithdetails} \eqref{item: Main5}. Since $E \to E'$ is purely inseparable, we can lift the action of $\Aut_X$ to $E \times C$, where it normalizes the $G$-action. By the previous paragraph, we have $N_{\Aut_C}(\alpha(G))/C_{\Aut_C}(\alpha(G)) \cong \GG_a$ and therefore $M$ is isomorphic to a subquotient of $(\ZZ/2\ZZ)^2$, again by Theorem \ref{Mainwithdetails} \eqref{item: Main5}. We may assume that $E$ is given by the equation $y^2 + \lambda y = x^3$.
Choose $c,d \in k$ such that $c^3 = \lambda$ and $d^2 + \lambda d + \lambda^2 = 0$ and let $h_{E,c,d}$ be the corresponding automorphism of $E$ as in Remark \ref{rem: M2}. Then, by the calculations of the previous paragraph and by Remark \ref{rem: M2}, $\alpha_T \circ {\rm ad}_{h_{E,c,d}} = {\rm ad}_{h_{C,c'}} \circ \alpha_T$, where $h_{C,c'}$ is a substitution 
$
h_{C,c'}: t \mapsto t + c'
$
with $c'^4 + \lambda c'^2 = c$. Therefore the automorphisms $(h_{E,c,d},h_{C,c'})$ of $E \times C$ descend to $X$. The three different values of $c$ yield three distinct non-trivial elements of $M$, so $M \cong (\ZZ/2\ZZ)^2$.

\end{enumerate}

\vspace{4mm}
This finishes the calculation of the groups $C_{\Aut_E}(G)/E, C_{\Aut_C}(\alpha(G))/\alpha(G),M,$ and thus also of the full automorphism group schemes for all bielliptic and quasi-bielliptic surfaces in all characteristics.
The results are summarized in Table \ref{bielliptic}, Table \ref{quasibielliptic3} and Table \ref{quasibielliptic2}.


\newcommand{\etalchar}[1]{$^{#1}$}


\begin{thebibliography}{ABD{\etalchar{+}}66+++}

\bibitem[ABD{\etalchar{+}}66]{SGA3}
M.~Artin, J.-E.~Bertin, M.~Demazure, A.~Grothendieck, P.~Gabriel, M.~Raynaud,
  and J.-P. ~Serre.
\emph{Sch\'emas en groupes},
S\'eminaire de G\'eom\'etrie Alg\'ebrique du Bois Marie (SGA3),
Institut des Hautes
  \'Etudes Scientifiques, Bures-sur-Yvette,
  1963--1966.

\bibitem[BdF10]{Bagnera}
G.~Bagnera and M.~de~Franchis,
\emph{Le nombre $\rho$ de Picard pour les surfaces hyperelliptiques},
Palermo Rend. {\bf 30} (1910), 185--238.

\bibitem[BM90]{BennettMiranda}
C.~Bennett and R.~Miranda,
\emph{The automorphism groups of the hyperelliptic surfaces},
Rocky Mountain J. Math. {\bf 20} (1990), no.~1, 31--37.


\bibitem[BM76]{BombieriMumford3}
E.~Bombieri and D.~Mumford,
\emph{Enriques' classification of surfaces in char. $p$. III},
Invent. Math. {\bf 35} (1976), 197--232.

\bibitem[BM77]{BombieriMumford2}
\bysame,
\emph{Enriques' classification of surfaces in char. $p$. II},
in: \emph{Complex Analysis and Algebraic Geometry} (A Collection of Papers Dedicated to K. Kodaira), pp. ~23--42, Cambridge University Press, Cambridge, 1977.

\bibitem[Bri11]{BrionTorsor}
M.~Brion,
\emph{On automorphism groups of fiber bundles},
Publ. Mat. Urug. {\bf 12} (2011), 39--66.

\bibitem[Bri18]{BrionNotes}
\bysame,
\emph{Notes on automorphism groups of projective varieties},
Lecture notes for the School and Workshop on Varieties and Group Actions, Warsaw, 2018.
Available at \url{http://www-fourier.univ-grenoble-alpes.fr/~mbrion/autos_final.pdf}

\bibitem[BSU13]{Brion}
M.~Brion, P.~Samuel, and V. Uma,
\emph{Lectures on the structure of algebraic groups and geometric applications},
CMI Lecture Series in Mathematics, vol. 1, Hindustan Book Agency, New Delhi; Chennai Mathematical Institute (CMI), Chennai, 2013.

\bibitem[Lan79]{Langquell}
W.~E. Lang,
\emph{Quasi-elliptic surfaces in characteristic three},
Ann. Sci. Éc. Norm. Supér. {\bf 12} (1979), no.~4, 473--500.

\bibitem[MO68]{MatsumuraOort}
H.~Matsumura and F.~Oort,
\emph{Representability of group functors, and automorphisms of algebraic
  schemes},
Invent. Math. {\bf 4} (1967/1968), 1--25.

\bibitem[MFK94]{MumfordGIT}
D.~Mumford, J.~Fogarty, and F.~Kirwan,
{\em Geometric invariant theory}, Ergebnisse der
  Mathematik und ihrer Grenzgebiete (2), vol. 34,
Springer-Verlag, Berlin, 1994.

\bibitem[Pop16]{Popov}
V.~L. Popov,
{\em Birational splitting and algebraic group actions},
European Journal of Mathematics {\bf 2}, 283--290 (2016).

\bibitem[Sil09]{Silverman}
J.~H. Silverman,
\emph{The arithmetic of elliptic curves}, 
Graduate Texts in Mathematics, vol. 106,
Springer, Dordrecht, 2009.

\end{thebibliography}
\end{document}